\documentclass[twoside,11pt,reqno]{amsart}
\date{\today}

\usepackage[latin1]{inputenc}
\usepackage{pstricks}
\usepackage{enumerate}	
\usepackage{comment}
\usepackage{graphicx,color}
\usepackage{caption}
\usepackage[english]{babel}
\usepackage[toc,page]{appendix}
\usepackage{amssymb, mathrsfs, amsfonts, amsmath,amsthm}
\usepackage{amsbsy, hyperref}
\usepackage{latexsym}
\usepackage{booktabs}
\usepackage{tikz}
\usepackage{eurosym} 
\usepackage{newfloat}
\usepackage{hyperref}

\usepackage[alphabetic, msc-links]{amsrefs}

\usepackage[normalem]{ulem}

\definecolor{r}{rgb}{.9,0.1,.3}

\numberwithin{equation}{section}
\newtheorem{Theorem}{Theorem}[section]
\newtheorem{Corollary}[Theorem]{Corollary}
\newtheorem{Proposition}[Theorem]{Proposition}
\newtheorem{Definition}[Theorem]{Definition}

\newtheorem{Claim}[Theorem]{Claim}
\newtheorem{Lemma}[Theorem]{Lemma}
\newtheorem*{lema*}{Lemma}

\newtheorem{Remark}[Theorem]{Remark}

\newcommand{\R}{\mathbb{R}}

\newcommand{\n}{\mathbb{N}}
\newcommand{\z}{\mathbb{Z}}

\newcommand{\sph}{\mathbb{S}}

\newcommand{\ee}{\mathbf e}
\newcommand{\p}{\partial}

\newcommand{\dist}{\operatorname{dist}}

\newcommand{\graph}{\operatorname{Graph}}

\newcommand{\width}{\operatorname{width}}

\definecolor{pp}{rgb}{.5,0,.8}

\definecolor{rr}{rgb}{.8,0,.3}
\definecolor{caca}{rgb}{0.137,0.3,0.541}

\newcommand{\Ff}{\mathcal F}

  \newcommand{\Ss}{\mathbf{S}}
 
 \newcommand{\RR}{\mathbf{R}}  
 \newcommand{\ZZ}{\mathbf{Z}}  
 \renewcommand{\SS}{\mathbf{S}}  

\usepackage{amsthm}

\input epsf
\def\begfig {
\begin{figure}
\small }
\def\endfig {
\normalsize
\end{figure}
}

\keywords{Mean curvature flow, translators, self-translating solitons, entropy, minimal surfaces, Morse-Rad\'o theory.}

\author[E.S. Gama]{E.S. Gama}
\address[Eddygledson Souza Gama]{
  Departamento de Matem\'atica, Centro de Ci\^encias Exatas e da Natureza,
  Universidade Federal de Pernambuco, 50670-901 Recife, Pernambuco, Brazil.
}
\email{eddygledson.gama@ufpe.br}

\author[F. Mart\'in]{F. Mart\'\i{}n}
\address[Francisco Mart\'in]{
  Departamento de Geometr\'\i{}a y Topolog\'\i{}a, 
  Instituto de Matem\'aticas IMAG,
  Universidad de Granada,
  18071 Granada, Spain.
}
\email{fmartin@ugr.es} 

\author[N.M. M{\o}ller]{N.M. M{\o}ller}
\address[Niels Martin M{\o}ller]{
  Copenhagen Center for Geometry and Topology (GeoTop), Department of Mathematical Sciences,
  University of Copenhagen,
  DK-2100 Copenhagen, Denmark.
}
\email{nmoller@math.ku.dk}

\subjclass[2020]{53A10, 53E10 (49Q05, 53C42).}

\thanks{F. Mart\'in is supported by the MICINN grant PID2020-116126-I00,  by the Regional Government of Andalusia and ERDEF grant PY20-01391 and by the IMAG--Maria de Maeztu grant CEX2020-001105-M/AEI/10.13039/501100011033.\\\indent{}N.M. M\o{}ller is supported by DFF Sapere Aude 7027-00110B, Carlsberg Foundation CF21-0680 and Danmarks Grundforskningsfond CPH-GEOTOP-DNRF151.}

\title[Finite entropy translating solitons in slabs]{Finite entropy translating\\ solitons in slabs}

\begin{document}
\maketitle

\begin{abstract}
We study translating solitons for the mean curvature flow, $\Sigma^2\subseteq\R^3$ which are contained in slabs, and are of finite genus and finite entropy.

As a first consequence of our results, we can enumerate connected components of slices to define asymptotic invariants $\omega^\pm(\Sigma)\in\mathbb{N}$, which count the numbers of ``wings''. Analyzing these, we give  a method for computing the entropies $\lambda(\Sigma)$ via a simple formula involving the wing numbers, which in particular shows that for this class of solitons the entropy is quantized into integer steps.

Finally, combining the concept of wing numbers with Morse theory for minimal surfaces, we prove the uniqueness theorem that if $\Sigma$ is a complete embedded simply connected translating soliton contained in a slab with entropy $\lambda(\Sigma)=3$ and containing a vertical line, then $\Sigma$ is one of the translating pitchforks of Hoffman-Mart\'i{}n-White \cite{HMW22-1}. 
\end{abstract}

\section{Introduction}

A smooth orientable surface $\Sigma^2$ immersed in $\R^3$ is called a translating soliton of the mean curvature flow if the mean curvature vector satisfies
\begin{equation}\label{TS-Equation-1}
\vec{H}=\ee_3^\perp,
\end{equation}
meaning that it moves ``upwards'' in the $\ee_3$-direction at unit speed. Consequently, if $N:\Sigma\to\mathbb{S}^2$ denotes the Gauss map of $\Sigma,$ then
\begin{equation}\label{TS-Equation-2}
H=-\langle\ee_3,N\rangle,
\end{equation}
where $H$ denotes the scalar mean curvature of $\Sigma.$ 

In this paper, we will be studying complete embedded translating solitons in $\R^3$ of finite entropy, so let us first recall the definition of the entropy of a surface in $\R^3$.

Following \cite{CM12} (See also \cite{Chi20}[Appendix A]), given $\Sigma^2\subseteq\R^3$ any surface, $x_0\in\R^3$ and $s_0>0$ we define
\[
F_{x_0,s_0}[\Sigma]=\frac{1}{4\pi s_0}\int_\Sigma e^{-\frac{|x-x_0|^2}{4s_0}}{\rm d}\Sigma \in (0,+\infty].
\]

\begin{Definition}
The entropy of $\Sigma^2\subseteq\R^3$, denoted by $\lambda(\Sigma)$, is: 
\[
\lambda(\Sigma)=\sup_{x_0\in\R^3,s_0>0}F_{x_0,s_0}[\Sigma]\in (0,+\infty].
\]

We say that a surface $\Sigma$ has finite entropy if $\lambda(\Sigma)<+\infty.$
\end{Definition}
A complete immersed translator with finite entropy is proper in $\R^3$ \cites{Chi20, KP22}.
Recent investigations  \cites{AW94, CSS07, DDN17, HMW22-1, HMW22-2, HMW24, KKM18, Ngu09}  indicate that the number of families of complete translators with finite entropy is vast. This means that in order to obtain any classification results, we have to introduce additional hypotheses to further restrict the class of complete translators under consideration. 

Thus, in this paper we are going to consider the case of complete translating solitons with finite entropy which are confined to slabs of $\R^3$. This geometric condition is first of all natural in the light of the convex hull classification theorem due to F. Chini and N. M\o{}ller \cite{CM19}-\cite{CM21}, from which we know that the projection in the $\ee_3$-direction of the convex hull of $\Sigma$ must be either a straight line, a strip between two parallel lines, a halfplane in $\R^2$, or the entire plane $\R^2$. The slab condition is also closely related to the notion of (non)collapsedness for mean convex flows. In fact, in the case of translating solitons with $H>0$, the property of being collapsed is equivalent to the slab condition, as a trivial consequence of the classification of complete translating graphs \cite{HMW22-1}. In this paper we therefore also view the slab condition as a version of collapsedness which is well-defined in our setting of not necessarily mean convex flows.
 
 The classification of the projection of the convex hull  allows, a priori, the possibility that translating solitons could be contained in slanted slabs. However, an easy first application of the maximum principle rules this out (see Proposition \ref{prop:vertical}.)

Up to rotations around the $z$-axis and horizontal translations (all of them preserve the velocity vector of the flow), we can thus assume that the slab is
\[
\mathcal{S}_w:= \{(x,y,z) \in \R^3 \; :\; |x|<w\}, \quad \mbox{for some $w>0$.}
\]
\begin{Definition}[Width]
If $\Sigma^2\subseteq\mathbb{R}^3$ is a complete connected translator contained in a slab, then we define the width of $\Sigma$ as the infimum of the numbers $2w$, where $w$ is given as in the previous paragraph. 
Then we write $\width (\Sigma)=2 w.$

If $\Sigma$ is not contained in any slab, we define
$\width (\Sigma)=+\infty.$
\end{Definition}

\begin{Remark}
From the preceding definitions it is clear that having $\width(\Sigma)=0$ implies $\Sigma=\{x=0\}$. In that special case, all of the results of this paper hold trivially true, and thus we shall from now on always assume that $\width(\Sigma)>0.$
\end{Remark}

If we impose the extra hypothesis of having finite genus, then we can prove that the entropy must always be a positive integer. 
Let $\Sigma$ be a complete, embedded translator with $\lambda(\Sigma)<\infty$ and ${\rm genus}(\Sigma)<\infty.$ Assume $\Sigma \subset \mathcal{S}_w$, for some $w>0.$ Then we 
 are also able to prove that outside a vertical cylinder (which contains the topology of $\Sigma$) $\Sigma$ consists of a disjoint union of finitely many non-compact, simply connected translators with boundary, that we call the wings of $\Sigma.$
 
 We shall demonstrate that there are two kinds of wings: the ones that are graphs over the $(y,z)$-plane and the ones that are bi-graphs over the $(y,z)$-plane. Wings of the first kind are called {\em planar wings} and wings of the second kind are called {\em wings of grim reaper type.}
 \begin{Theorem}\label{lambda-formula}
Let $\Sigma^2\subseteq\mathbb{R}^3$ be a complete embedded translator so that the width, the entropy and the genus are finite. If $\omega_P(\Sigma)$ represents the number of planar wings and $\omega_G(\Sigma)$ represents the number of wings of grim reaper type, then $\omega_P(\Sigma)+2\, \omega_G(\Sigma)$ is even and 
 \begin{equation}
     \lambda(\Sigma)= \frac 12 \left( \omega_P(\Sigma)+2 \, \omega_G(\Sigma)\right).
 \end{equation}
\end{Theorem}

In the course of proving Theorem \ref{lambda-formula}, we establish that the translated surfaces $\Sigma + t \mathbf{e}_3$ admit subsequential limits as $t_k \to \pm\infty$, see Proposition \ref{weakly-infinity-limit} below. These limits are as follows:\begin{enumerate}[a)]
    \item As $t_k \to -\infty$, the limits consist of $\lambda(\Sigma)$ vertical planes parallel to the boundary of the slab.
    \item As $t_k \to +\infty$, the limits consist of a collection of $\frac{1}{2}\omega_P(\Sigma)$ vertical planes, also parallel to the boundary of the slab.
\end{enumerate}
In both cases, the planes could appear with multiplicity, see Corollary \ref{components-estimate} below. In the forthcoming work \cite{GMM25}, the authors demonstrate (among other results) that the limits in a) and b) are not merely subsequential, but that in fact this convergence happens as true smooth limits, independently of the chosen subsequences. That is, as $t \to \pm \infty$, the translated surfaces $\Sigma + t \mathbf{e}_3$ converge smoothly in compact subset of $\R^3$ to the respective collections of vertical planes, with multiplicity.

F. Chini \cite{Chi19}, \cite{Chi20} proved that if   $\Sigma^2\subseteq\mathbb{R}^3$ is a complete, embedded translator  which is simply connected, contained in a slab and $\lambda(\Sigma)<3$, then $\Sigma$ is mean convex. In particular, by the classification in \cite{HIMW19} (see Theorem \ref{Classification} below), $\Sigma$ is either a vertical plane, a $\Delta$-wing or a tilted grim reaper. As a first demonstration of the techniques of the present paper, in Section \ref{l<3} we give a new proof of Chini's theorem within our framework.

The main result of this paper consists in extending the previous classification results to the case $\lambda(\Sigma)<4$ under additional hypotheses, as follows:
\begin{Theorem}
 Let $\Sigma^2\subseteq\mathbb{R}^3$ be a simply connected complete embedded translator of finite width and $\lambda(\Sigma)=3.$ Assume that $\Sigma$  contains a vertical line.     
Then $\Sigma$ is one of the pitchfork translators constructed by Hoffman, Mart\'i{}n and White \cite{HMW22-1} (see Fig.
 \ref{fig:pitchfork}).
\end{Theorem}

Owing to \cite{KS19} (see \cite{IR19} as well), this also classifies all stable, w.r.t the ambient metric in \eqref{g-Ilmanen}, translators of finite width and low entropy, under additional mild assumptions:

\begin{Corollary}
The pitchforks are the unique complete embedded and stable translators of finite width and $\lambda(\Sigma)=3$ which contain a vertical line.
\end{Corollary}

Another interesting related type of translators that we are going to study in this paper is the class of graphs in general direction contained in slabs, which we will soon show is in fact a subclass of the more general translating solitons which we will be considering, namely having also finite entropy and, being simply connected, representing a special case of finite genus. The examples from this class also play an important role in motivating the more general theorems which we will be proving.

\begin{Definition}\label{def:graphs}
Fix any unit vector $v\in\R^3$. Let $\Omega\subseteq[v]^\bot$ be a planar domain. We say that a smooth function $u:\Omega\to\R$ is a translator provided that
\[
\Sigma:={\rm Graph}_{(v,\Omega)}[u]:=\left\{p+u(p) \, v:\;p\in\Omega\right\}\subseteq \R^3
\]
is a translating soliton as in \eqref{TS-Equation-1}. We will refer to such $\Sigma$ as a translating graph.
\end{Definition}
 Our basic assumption in this paper is that the translating graphs ${\rm Graph}[u]$ are complete in $\R^3$, and by a slight abuse of language we will then say that $u$ is complete.

If we imposed that the direction of the graph $v$ coincided with the velocity of the flow $\mathbf{e}_3$, there would be nothing to prove, as such complete graphs have already been completely classified. Namely, when $v=\mathbf{e}_3$, the complete connected translating graphs in the direction of $v$ are: the bowl soliton, the family of tilted grim reaper surfaces and the family of $\Delta$-wings (see Theorem \ref{Classification}.)

However, if $v\neq \mathbf{e}_3$, as we have already mentioned,  recent investigations  \cites{HMW22-1, HMW22-2, HMW24}  indicate that the number of families of complete graphs is very large. This means that in order to obtain any classification results, we have to introduce additional hypotheses to further restrict the class of complete graphs under consideration.

If $\Sigma=\left\{p+u(p) \, v:\;p\in\Omega\right\}$, for $\Omega$ an open planar domain in $[v]^\perp$, is a graph contained in the slab $\mathcal{S}_w$, then we are going to distinguish two cases:
\textbf{Case I}: $v\notin\Ss^2\cap\{x=0\}$; \textbf{Case II}: $v \in \Ss^2 \cap \{x=0\}.$

Case I, which at face value might seem to deal with a larger set of examples, is in fact easier than Case II, and will be handled already by Proposition \ref{Plane-Case} below, which asserts that any connected translating graph in Case I must be a plane $\{x=x_0\}$.
Case II, meaning $v\in\Ss^2\cap\{x=0\}$, is not so elementary and the number of known examples is very large \cites{HMW22-1, HMW22-2, HMW24}. In Section \ref{Background}, we will in particular prove that any complete translating graph contained in a slab is simply connected and has finite entropy. As a consequence of Theorem \ref{lambda-formula}, the entropy is also always a positive integer for such graphs. We would like to point out that the scherkenoids constructed by Hoffman-Mart\'in-White \cite{HMW22-1} are examples of complete connected translating graphs which do not lie in any slab and they have infinite entropy. On the other hand, the bowl soliton constructed by Altschuler-Wu \cite{AW94} (see also Clutterbuck-Schn{\"u}rer-Schulze in \cite{CSS07}) is an example of a complete connected vertical translating graph that does not lie in a slab and it has finite entropy, which is no longer an integer, but has the value $\lambda(\Sigma) = \lambda(\mathbb{S}^1\times \R) = \sqrt{2\pi/e}$.

We furthermore conjecture that the plane and the bowl soliton (a vertical graph) are the unique examples of translating graphs of finite entropy whose domains of definition are not contained in planar strips.

As a consequence of the classification due to \cite[Theorem 1]{Chi20} and Corollary \ref{entropy-estimate}, we also classify the graphs of small width, in Corollary \ref{Classification-slab-3pi} in Section \ref{l<3}.

The main techniques used in this paper come from the fact, observed first by T. Ilmanen \cite{Ilm94}, that translators can be seen as minimal surfaces in $\R^3$ endowed with the conformal metric
\begin{equation}\label{g-Ilmanen}
g=e^z \langle \cdot,\cdot \rangle_{\R^3}.
\end{equation}

Firstly, this means that we are able to use White's general compactness results for minimal $2$-surfaces in Riemannian $3$-manifolds under uniform bounds on genus and area \cite{Whi18}.

We can also make use of the so-called Morse-Rad\'o theory for minimal surfaces \cite{HMW23}. Roughly speaking, Morse-Rad\'o theory studies the way in which a foliation by minimal surfaces intersects a fixed minimal surface $M$ inside a Riemannian $3$-manifold. The structure of the resulting family of curves in $M$ is determined by the topology of $M$ and the behaviour of the mentioned foliation along $\partial M$. Given a  complete translator $\Sigma$, contained in a slab, with finite genus and finite entropy, the Morse-Rad\'o theory allows us to understand the finer nature of the set $$\mathcal{H}:=\{p \in \Sigma \;: \; H(p)=0\}, $$
and its image under the Gauss map (Section \ref{H-section}). This detailed insight will be crucial in the classification results of Section \ref{pitchforks-section}, as well as in Section \ref{sec:chini}, where we reprove the already known result by Chini, that finite width, 1-connectedness and $\lambda(\Sigma) < 3$ together imply that $\mathcal{H}=\varnothing$.

{\bf Acknowledgements.} We thank D. Hoffman and B. White for valuable conversations and suggestions about this work. Part of this work was done when E. S. G. was a professor at Universidade Federal Rural do Semi-\'Arido, Cara\'ubas, and he thanks the institution for the fruitful working environment.

\section{Examples of translating graphs}\label{Examples}

The purpose of this part is to introduce the key examples of translating graphs that will be used throughout the paper. The more advanced examples presented here were constructed in \cite{HIMW19}, \cite{HMW22-1} and \cite{Ngu09}. 

\subsection*{Tilted grim reaper surfaces} The first nontrivial examples of translating graphs, apart from the planes (i.e. planes parallel to $\ee_3$), are given by the family of tilted grim reaper surfaces $\{\mathcal{G}_\zeta\}_{\zeta\in[0,\pi/2)\cup(\pi/2,\pi)}$ defined by
\[
\mathcal{G}_\zeta=\left\{
\begin{array}{cc}
 & x\in\left(-\frac{\pi}{2|\cos\zeta|},\frac{\pi}{2|\cos\zeta|}\right)\ \\
\Big(x,y,-\displaystyle\frac{\log\cos(x\cos\zeta)}{\cos^2\zeta}-y\tan\zeta\Big): & {\rm and} \\
 &y\in\R \\
\end{array}
\right\}.
\]

Indeed, for all $\eta\neq\zeta$, $\mathcal{G}_\zeta$ is a graph of a smooth function in the direction of $v_\eta=\sin\eta\ee_3+\cos\eta\ee_2.$
The surface $\mathcal{G}_{0}$ is called the standard grim reaper surface. 

\subsection*{$\Delta$-wings}
Based on previous works of Wang \cite{Wan11} and Spruck and Xiao \cite{SX20}, Hoffman et al. \cite{HIMW19} classified the complete translating graphs over domains in $(x,y)-$plane
as follows:

\begin{Theorem}[Classification of complete vertical graphs, \cite{HIMW19}]\label{Classification}
Given $w>\pi/2$, there is a unique (up to translations, and rotations preserving the direction of translation) complete, strictly convex translator 
$
   u^w: \left(-w,w\right)\times\R\to\R.
$
Up to translations, and rotations preserving the direction of translation, the only other complete vertical translating graphs in $\R^3$ are the tilted grim reaper surfaces, and the bowl soliton.
\end{Theorem}
\begin{Remark}
The graph of the function $u^w$ is called the $\Delta$-wing of width $2w$. We also note that by \cite{SX20} the assumption of strict convexity in Theorem \ref{Classification} is a direct consequence of being mean convex, which by \eqref{TS-Equation-2} is automatic for vertical graphs.
\end{Remark}

The next two subsections are devoted to describing some families of complete, simply connected, translators introduced by Hoffman, Mart\'i{}n and White in \cite{HMW22-1}. By Theorem \ref{Classification}, they cannot be graphs in the direction of translation, however they will be graphs into other directions.

\subsection{Pitchforks} \label{subsec:pitchfork}
 Pitchforks were introduced by Hoffman, Mart\'i{}n and White in \cite{HMW22-1}. In this case, a fundamental piece of the surface is a graph over the strip in the $(x,y)$-plane
 $$u: \left( 0, w\right) \times \RR \rightarrow \RR, \quad \mbox{where $w \geq \pi.$}$$
 with the following boundary values:
\begin{equation}
    \label{eq:pitchfork}
u(x,y)=
\begin{cases}
-\infty &\text{for $x=0$ and $y<0$, } \\
\infty &\text{for $x=0$ and $y>0$, } \\
\infty &\text{for $x=w$}.
\end{cases}
\end{equation}
Then $\graph[u]$ is a translator with boundary which contains the $z$-axis. By reflecting this graph
about the $z$-axis we obtain a complete (without boundary), simply connected, translating soliton (see Fig. \ref{fig:pitchfork})
of finite entropy. 
\begin{Remark} \label{pitchfork-uniqueness}
   Mart\'in, S\'aez, Tsiamis and White \cite{MSTW24} have proved that any two functions solving \eqref{eq:pitchfork} differ by a constant.
\end{Remark}

\begin{figure}[htbp]
\begin{center}
\includegraphics[height=.29\textheight]{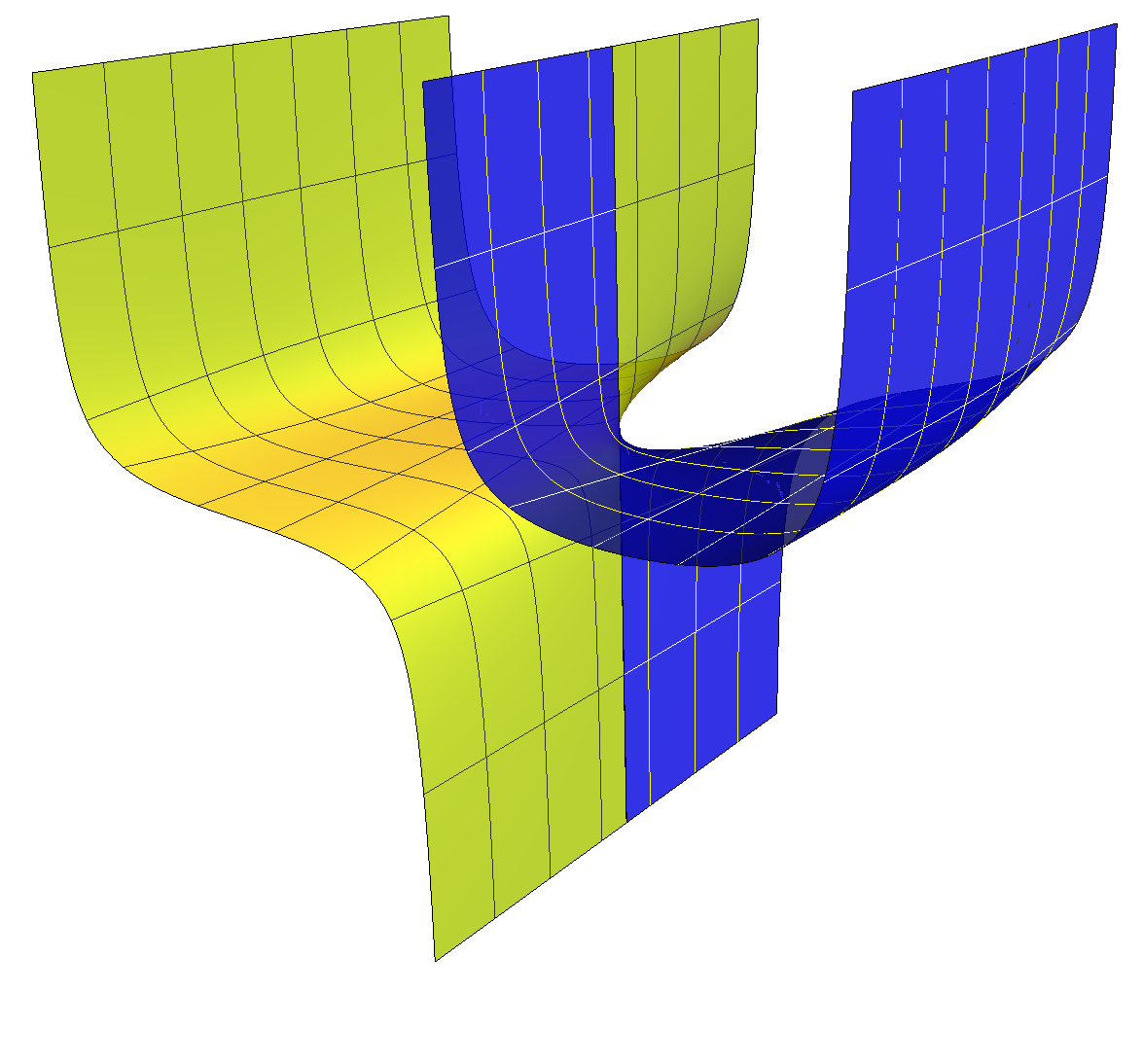}
\end{center}
\caption{A fundamental piece  (in blue) of width $\pi$ of the pitchfork of width $2\pi$. 
The whole surface is obtained by reflecting the fundamental piece with respect to the $z$-axis (left).} \label{fig:pitchfork}
\end{figure}

If $w>\pi$, the complete pitchfork of width $2w$ is a graph over a strip in the $(x,z)$-plane:
$$ \Omega = \left(-w, w \right) \times \{0\}\times \R.$$
This can also be visualized in Figure \ref{fig:pitchfork}, since the normal never points along the $(x,z)$-plane (which is roughly in the plane of the paper) and a rigorous proof of this can be found in Theorem 12.1 in \cite{HMW22-1}.

In the case $w=\pi$, we only that know that the corresponding pitchfork is a graph over some subdomain of the $(x,z)$-plane
$$ \Omega \subseteq \left(- \pi, \pi \right) \times \{0\}\times \R.$$
However, we do not know whether $\Omega$ coincides with the whole strip or if it is a proper subdomain.

\subsection{Semigraphical Translators}\label{semigraphical-section} When you remove the symmetry axis from pitchforks, they are graphs over a strip of the $(x,y)$-plane. This motivates the following 
\begin{Definition}

A translator is $M$ is called semigraphical if 
\begin{enumerate}[\upshape (1)]
\item $M$ is a smooth, connected, properly embedded submanifold (without boundary) in $\R^3$.
\item $M$ contains a nonempty, discrete collection of vertical lines.
\item $M\setminus L$ is a graph, where $L$ is the union of the vertical lines in $M$.
\end{enumerate} \end{Definition}
Suppose $M$ is a  semigraphical translator.  We may suppose w.l.o.g.
that $M$ contains the $z$-axis $Z$. Note that $M$ is invariant under $180^\circ$ rotation about each line in $L$,
from which it follows that $L\cap\{z=0\}$ is an additive subgroup of $\R^2$. 
Semigraphical translators have been classified, in the following sense
\begin{Theorem}[\cite{HMW22-2}*{Theorem 34} and \cite{MSTW24}*{Theorem 1}]\label{semigraphical-theorem}
A semigraphical translator $M$ in $\R^3$ is one of the following:
\begin{enumerate}[(1)]
\item a (doubly periodic) Scherk translator;
\item a (singly periodic) Scherkenoid;
\item a (singly periodic) 
  translating helicoid;
\item a pitchfork; 
\item a (singly periodic) trident.
\end{enumerate}
Up to isometries of $\R^3$, 
\begin{enumerate}[\upshape (i)]
\item 
translating helicoids and 
 pitchforks 
are uniquely determined by 
a single parameter, the width of a
fundamental region,
\item
tridents are uniquely determined by 
a single parameter, the period, and
\item
Scherk translators and Scherkenoids are uniquely determined by two parameters, the width and an angle.
\end{enumerate}

\end{Theorem}

\begin{figure}[h]
{\includegraphics[width=3.8cm]{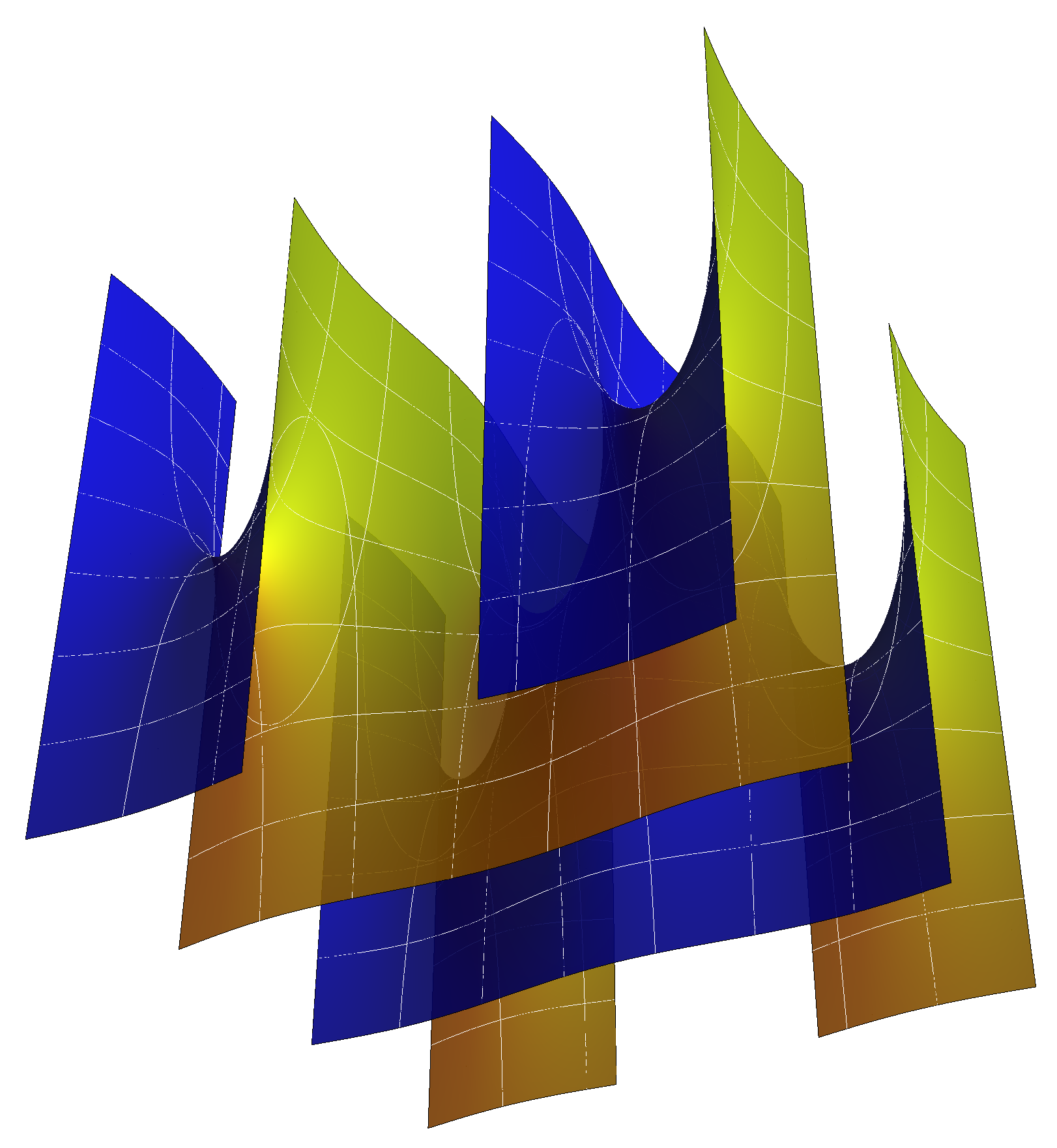}}
{\includegraphics[width=4cm]{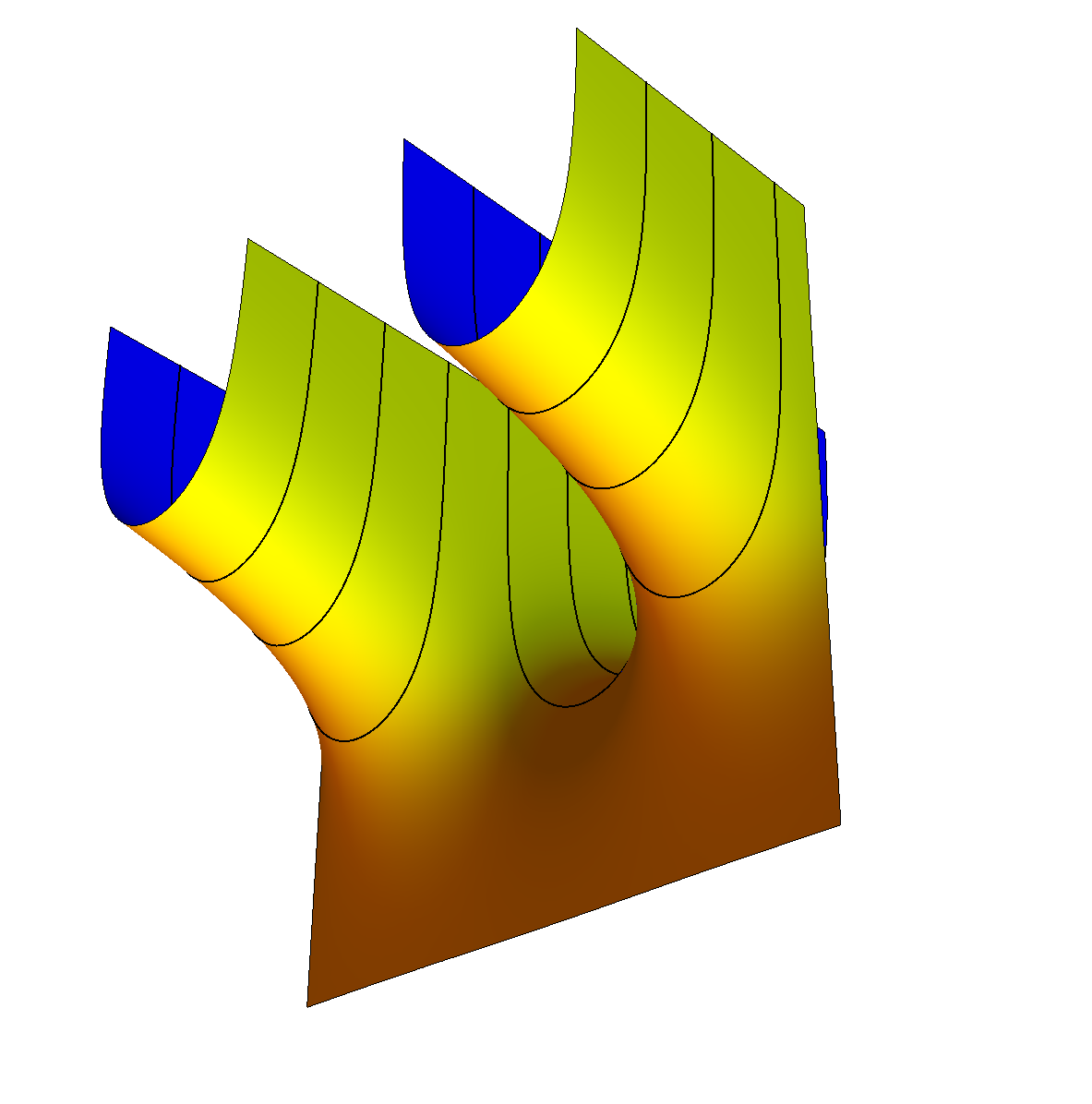}}
{\includegraphics[width=4cm]{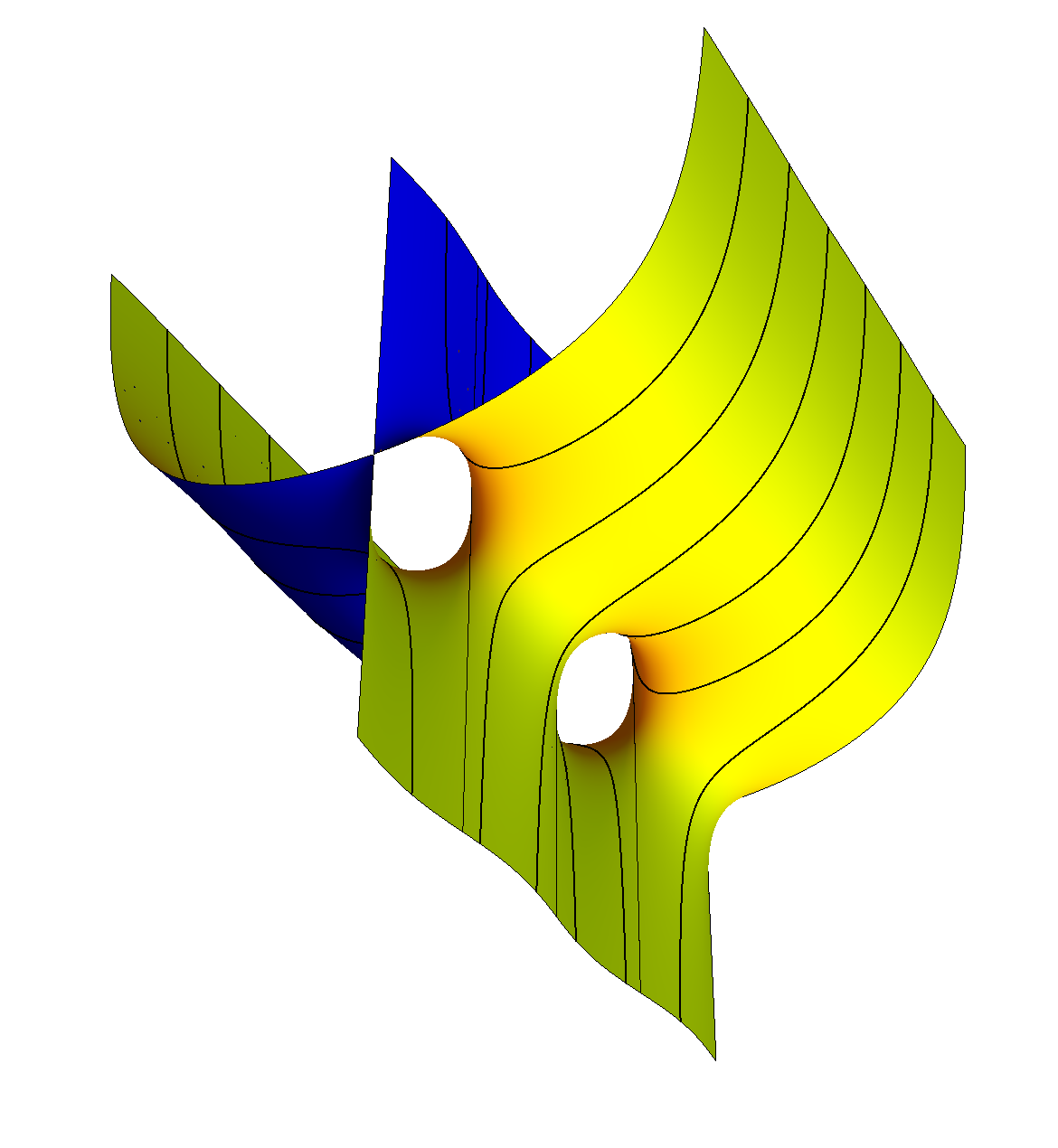}}
{\includegraphics[width=3.8cm]{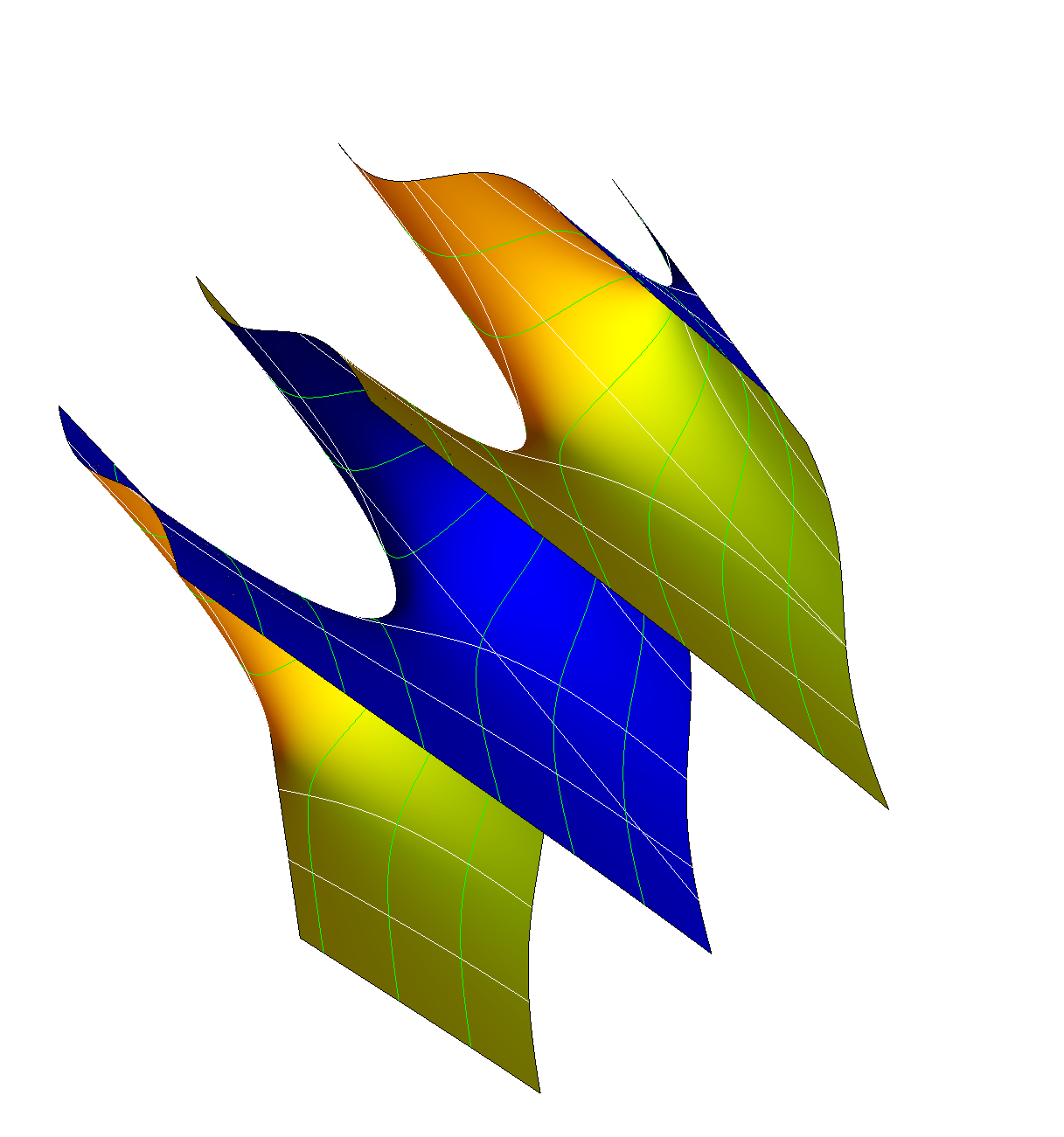}}
{\includegraphics[width=4cm]{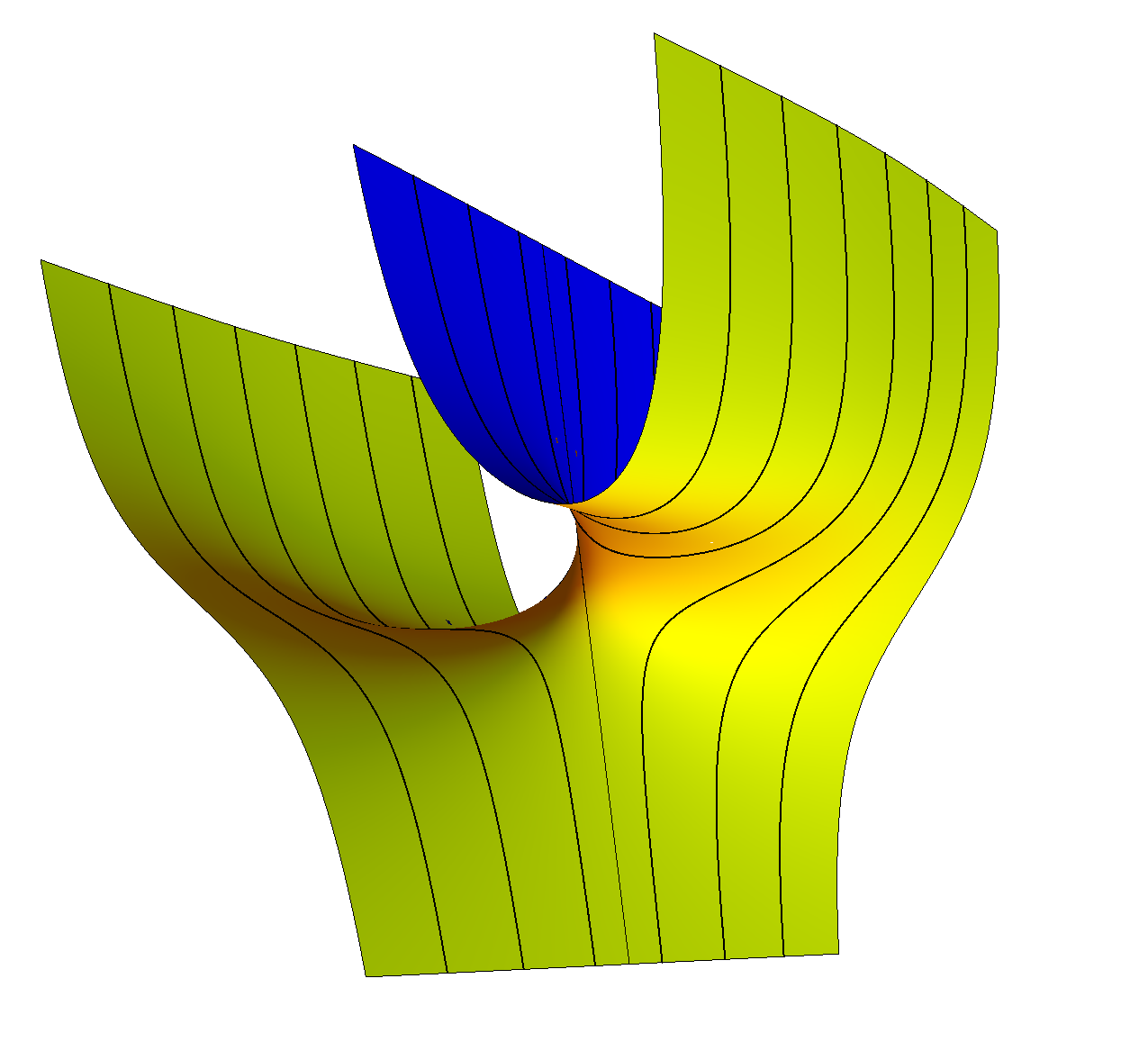}}
\caption{Semigraphical translators.
Row 1 (left to right): Scherk translator, scherkenoid, and trident.
Row 2: Helicoid-like translator and pitchfork.}
\label{sequence2}
\end{figure}

\subsection*{Notation}\label{sec:notation}
Throughout the paper, we are going to use the following notations related to slicing and decomposing the surfaces: For a given set $A$ in $\R^3,$ we define the sets:
\[
A_{+}(t):=\{(x,y,z)\in A\;:\;x\geq t\},A_{-}(t):=\{(x,y,z)\in A\;:\;x\leq t\},
\] 
\[
A^{+}(t):=\{(x,y,z)\in A\;:\;y\geq t\},A^-(t):=\{(x,y,z)\in A\;:\;y\leq t\},
\] 
\[
A(s,t):=\{(x,y,z)\in A\;:\;s<y< t\}.
\]
and
\[
A[s,t]:=\{(x,y,z)\in A\;:\;s\leq y\leq t\}.
\]

We also make extensive use of the three families of planes:
\begin{eqnarray*}
\pi_x(t) &:= &\{(x,y,z) \in \R^3 \; : \;x=t\}\\ 
\pi_y(t) &:= &\{(x,y,z) \in \R^3 \; : \;y=t\}\\
\pi_z(t) &:= &\{(x,y,z) \in \R^3 \; : \;z=t\}.\end{eqnarray*}
 and the maps
 \begin{eqnarray*}
\Pi_x(x,y,z) &:= &(y,z)\\ 
\Pi_y(x,y,z) &:= &(x,z)\\
\Pi_z(x,y,z) &:= &(x,y).
\end{eqnarray*}

\section{Preliminaries and Background}\label{Background}

In this section we compile some of the general results about translating solitons which we will be needing throughout the paper. See also Appendix \ref{Results-g-area}, where we collect some more of their consequences and technical parts of the proofs.

As we mentioned in the introduction, that translators are minimal surfaces with respect to Ilmanen's metric $g=e^{z}\langle\cdot,\cdot\rangle_{\R^3}$ was first proven in \cite{Ilm94}.

\begin{Definition}
From now on, we will therefore also refer to translators as $g$-minimal surfaces. The associated geometric quantities will be indexed accordingly, e.g. $\mathcal{H}^2_g(\cdot)$ will denote the two-dimensional Hausdorff measure associated to Ilmanen's metric on $\R^3$.
\end{Definition}

\begin{Definition}
Let $\Omega$ be a subset of $P_v=[v]^\perp.$ The set \[\Omega\times_v\R:=\{p+tv\;:\; p\in \Omega\ {\rm and}\ t\in\R \}\] is called the cylinder over $\Omega$ in the direction of $v.$ 
\end{Definition}

Stability of translating vertical graphs was observed in Shahriyari paper \cite[Theorem 2.5]{Sha15}. In fact we will be making extensive use of the stronger $g$-area minimizing properties of graphs in general direction via the next lemma, which is inspired in a general result by  B. Solomon \cite{Sol89}.
\begin{Lemma}\label{Area-minimizing}
Let $\Sigma={\rm Graph}_{(v,\Omega)}[u]$ be a complete connected translating graph. Then, $\Sigma$ is $g-$area minimizing in $\R^3.$
\end{Lemma}
\begin{proof}
See \cite[Corollary 1.1]{Sol89} or \cite[Appendix A]{HMW22-1} for the proof.
\end{proof}

Along the paper we will make use of the following compactness theorem which shall be a fundamental tool in the proof of several results.

\begin{Theorem}[Compactness] \label{theorem-compactness} Let $\Sigma$ be a complete, embedded  translating soliton with finite genus and finite entropy and let $\{p_i\}$ be a sequence in $\R^3$.

Define $\Sigma_i:= \Sigma-p_i$.  Then,
after  passing to a subsequence, $\{\Sigma_{i}\}_{i\in\n}$ has a limit which is either empty or a
smooth properly embedded self-translating surface $\Sigma_{\infty}$.
If the limit is not empty, the convergence is smooth
away from a discrete set denoted by $\operatorname{Sing}\subseteq \Sigma_\infty$. Moreover, for each connected
component $\Sigma'$ of $\Sigma_{\infty}$, either
\begin{enumerate}
\item [\rm (a)] the convergence to $\Sigma'$ is smooth everywhere with multiplicity $1$, or
\smallskip
\item [\rm (b)] the convergence is smooth, with some multiplicity greater than one, away
from $\Sigma'\cap\operatorname{Sing}$.
\end{enumerate}
\end{Theorem}

\begin{proof}
As the entropy $\lambda(\Sigma)<\infty,$ then $\Sigma$ has quadratic area growth (see, for instance \cite[Theorem 9.1]{Whi21}.) Then the sequence $\Sigma_i$ has uniform area bounds on compact subsets of $\R^3.$ Moreover, as the genus of $\Sigma$ is finite
then the genus of $\Sigma_i$ is also locally uniformly bounded. Hence, the proof of the theorem is a direct consequence of Theorem 1 in \cite{Whi18}.
\end{proof}

Next, we want to prove that if $\Sigma$ is a complete translating graph contained in a slab $\mathcal{S}_w$, then $\Sigma$ has finite entropy. 

We start by observing that Lemma \ref{Locally-Area} allows us to conclude the uniform estimate for the area of a complete connected translating graph inside a cube of length $w.$
\begin{Lemma}\label{cube-estimate}
Let $\Sigma$ be a complete connected translating graph of width $2w$ and $Q$ be a cube in $\R^3$ of side length $w.$ Then, there exists a constant $C>0$ (independent of $Q$ and $\Sigma$) such that 
\[{\rm Area}(Q\cap\Sigma)\leq C,\]
where ${\rm Area}(\cdot)$ indicates the area functional in $\R^3$ endowed with the Euclidean metric. 
\end{Lemma}

As an application of this result, we are going to prove that complete connected translating graphs have finite entropy. The ideas used in the proof of the next proposition are inspired by \cites{HMW24, Whi22}.
\begin{Proposition}\label{finiteentropy}
Let $\Sigma$ be a complete connected translating graph of width $2w$. Then $\Sigma$ has quadratic area growth. In particular, the entropy of $\Sigma$ is finite.
\end{Proposition}
\begin{proof}
Firstly, we assume that $R\geq w.$ Consider $p=(p_1,p_2,p_3)\in\R^3.$ Let $n\in\mathbb{N}$  so that
\begin{equation}\label{A-estimate}
    (n+1)w>R\geq nw.
\end{equation}
Then
\[
B_R(p)\subset\prod_{j=1}^3[p_j-(n+1)w,p_j+(n+1)w].
\]
since $\Sigma\subseteq \mathcal{S}_w,$ one obtains: 
\[
\Sigma\cap B_R(p)\subseteq\Sigma\cap \{[p_1-w,p_1+w]\times\prod_{j=2}^3[p_j-(n+1)w,p_j+(n+1)w]\}.
\]
Let $\Tilde{Q}=[p_1-w,p_1+w]\times\displaystyle\prod_{j=2}^3[p_j-(n+1)w,p_j+(n+1)w]$ and notice that this set is the union of $8(n+1)^2$ cubes of side length $w$. Therefore, Lemma \ref{cube-estimate} implies
\[
{\rm Area}(\Sigma\cap B_R(p))\leq  8(n+1)^2C
\]
and hence by \eqref{A-estimate}  one gets
\[
\frac{{\rm Area}(\Sigma\cap B_R(p))}{\pi R^2}\leq \frac{8(n+1)^2}{\pi n^2w^2}C=:\widetilde C,
\]
where $\widetilde C$ depends only on $w$.
Now, if $R<w,$ then Lemma \ref{cube-estimate} and the monotonicity formula ensure that 
\[
C_1=\sup_{\begin{array}{c}
p\in\R^3  \\
R\in(0,w)  \\
\end{array}
}\frac{{\rm Area}(B_R(p))}{\pi R^2}<+\infty.
\]
This completes the proof of the first part. The second part follows from \cite[Theorem 9.1]{Whi21}.
\end{proof}

The next result is essentially due to M. Eichmair and J. Metzger \cite[Appendix B]{EM16} and L. Shahriyari \cite{Sha15}. It describes the structure of the domains of complete translating graphs (see also  \cite{CM21}) and guarantees in particular that any complete translating graph is simply connected.
 
\begin{Lemma}\label{domain}
Let $u:\Omega\to\R$ be a complete connected translating graph in the sense of Definition \ref{def:graphs}, not necessarily contained in a slab. Then each connected component of $\partial \Omega\times_v\R$ is either a tilted grim reaper (thus not a graph in the $v$-direction) or a vertical plane containing $\ee_3$. Thus, $\Omega$ is unbounded and simply connected, and there are at most two parallel lines contained in $\partial \Omega$. 
\end{Lemma}

\begin{proof}
 By \cite[Appendix B]{EM16} we know that the sequence $\{\Sigma\,\pm\, n v\}_{n\in\mathbb{N}}$ converges smoothly on compact subsets, after passing to a subsequence, to $\Sigma^{\pm}_\infty$, where $\Sigma:={\rm Graph}_{(v,\Omega)}[u]$  and $\Sigma_\infty^{-}\cup\Sigma_\infty^{+}=\partial\Omega\times_v\R$. Thus each component of the limit translator is a connected ruled translator, which were classified in \cite[Corollary 2.1]{MSHS15}; they are either tilted grim reapers or vertical planes. 
 
Once we know that every connected component of $\partial\Omega$ is either an entire line or a complete grim reaper curve, which in particular are unbounded curves, it also follows from basic topological properties of $\R^2$ that the connected domain $\Omega$ is simply connected.
\end{proof}

\begin{Remark}
We would like to point out that all known examples in Lemma \ref{domain} are graphs over strips in $P_v$ or over the entire plane $P_v$.

We conjecture that the strips and planes are the unique examples of the planar domains which occur, so that also the half-plane is ruled out in Lemma \ref{domain}, as in case $v=\ee_3$ where it was already proven in \cites{HIMW19,SX20,Wan11}.
\end{Remark}

As a consequence of Theorem \ref{theorem-compactness} and Proposition \ref{finiteentropy}, we get the following extra information in the case of graphs.

\begin{Corollary}\label{Limit-graphs}
Let $\Sigma^2\subseteq\R^3$ be a complete translating graph in the direction of $v\in \sph^2$ and $\{p_k\}$ be a sequence in $\R^3$. 

Define $\Sigma_k:=\Sigma\, -\, p_k$. Then, up to passing to a subsequence, $\{\Sigma_k\}$ has a smooth limit on compact subsets of $\R^3$; $\Sigma_\infty = \Sigma_\infty^2\subseteq\mathbb{R}^3$, possibly empty. If $\Sigma_\infty$ is not empty, then each connected component of $\Sigma_\infty$ is either:
\begin{itemize}
    \item a translating graph in the direction of $v$, or
    \item a grim reaper cylinder tangent to $v$, or
    \item a vertical plane (i.e. containing $\ee_3$).
\end{itemize}
\end{Corollary}
\begin{proof} The convergence statement is a consequence of Theorem \ref{theorem-compactness} (taking into account that $\Sigma_k$ is simply connected, by Lemma \ref{domain}).

On the other hand, if we denote by $N_k$ the Gauss map of $\Sigma_k$, then $f_k=\langle N_k, v\rangle$ is a positive Jacobi function over $\Sigma_k$. If we label $f_\infty=\lim_k f_k$, then $f_\infty \geq 0$
and it is a Jacobi function over $\Sigma_\infty$. 

Consider $\Sigma'$ a connected component of $\Sigma_\infty$. Hence, by a local application of the strong maximum principle for the stability operator either $f_\infty|_{\Sigma'}>0$ everywhere or $f_\infty|_{\Sigma'} \equiv 0$. In the first case, $\Sigma'$ is a graph in the direction of $v$. In the second case, we have that $v$ is tangent to $\Sigma'$ and its Gauss curvature $K \equiv 0$. By Corollary 2.1 in \cite{MSHS15} we have that $\Sigma'$ is either a vertical plane, or a grim reaper cylinder (tangent to $v$). This concludes the proof.
\end{proof}

\subsection{Chini-type results}
In this part we collect some results from \cite{Chi20} which will be used throughout the paper.

A fundamental fact in the proof of several results in this paper will be the following theorem by F. Chini \cite{Chi19}, \cite[Theorem 10 and Corollary 11]{Chi20}. This theorem provides an important description of the asymptotic behavior of a complete translator of finite width $2 w>0$ when it approaches the boundary planes $\{x=\pm w\}$.
\begin{Theorem}{$($\cite[Th. 10]{Chi20}$)$} \label{Chini-theorem}
 Let $\Sigma^2$ be a properly immersed translator of $\R^3$ such that $\partial( {\rm Conv}(\Pi_z(\Sigma))) \neq \varnothing.$
 
 Then for every $q\in \partial ({\rm Conv}(\Pi_z(\Sigma)))$ and for every $\rho>0$ we have that
 $$ \sup_{\Sigma \cap B(q,\rho)} z =+\infty.$$
\end{Theorem}
\begin{Theorem}{$($\cite[Cor. 11]{Chi20}$)$.}\label{Chini-result}
Let $\Sigma^2\subseteq\R^3$ be a connected properly immersed translator of width $2w>0$. Then, for any plane $\pi$ parallel to $\ee_3$ and not parallel to $\ee_2$, there exist two distinct sequences $\{p_1^k\},\{p_2^k\}\subseteq\Sigma\cap\pi$ satisfying the following properties:
\begin{itemize}
    \item[a.]\(\displaystyle\lim_{k\to+\infty} \langle p^k_1,\ee_3\rangle=\lim_{k\to+\infty} \langle p^k_2,\ee_3\rangle=+\infty;\)
    \item[b.]\(\displaystyle\lim_{k\to+\infty} {\rm dist}(p^k_1,L_1)=\lim_{k\to+\infty}{\rm dist}(p^k_2,L_2)=0.\)
\end{itemize}
where the lines are $L_1:=\pi\cap\{x=-w\}$ and $L_2:=\pi\cap\{x=w\}$
\end{Theorem}

At this point, we would like to point out an important fact which will be useful later.
\begin{Remark} \label{re:pre}

Let $t\in\R$ so that $\{y=t\}$ is transverse to $\Sigma$, which holds for almost every $t$ by an application of Sard's theorem to the map which projects $\R^3$ to the $y$-axis. Using that $\Sigma$ is proper (see Proposition \ref{Proper-condition} in Appendix A), it can be deduced that $\Sigma\cap\{y=t\}$ is a union of proper curves on $\{y=t\}\cap\{-w<x<w\},$ where $2w$ is the width of $\Sigma.$

On the other hand, Chini's Theorem \ref{Chini-result} implies that there exist two arcs in $\Sigma\cap\{y=t\}$ (possibly both arcs are part of the same curve) that  we call $\gamma_1$ and $\gamma_2,$ so that $\gamma_i$ is sequentially asymptotic to the line $L_i:=\{y=t\}\cap\{x=(-1)^iw\}.$ Thus, if we consider an increasing sequence $\{\tau_n\}$ with $\tau_n\to+\infty$ as in Theorem \ref{Chini-result}, then the sequence $\{\Sigma-\tau_n\ee_3\}$ satisfies that any subsequential limit  of $\{\Sigma-\tau_n\ee_3\}$ must contain the planes $\{x=-w\}$ and $\{x=w\}.$ 
\end{Remark}

Using this, we are ready to classify the complete graphs in $\mathcal{S}_w$, $w\geq0$, when $v\notin \Ss^2\cap\{x=0\}.$
\begin{Proposition}\label{Plane-Case}
 Let $\Sigma$ be a complete connected translating graph of finite width in the direction of $v\notin \Ss^2\cap\{x=0\}.$ Then $\Sigma$ is a plane $\{x=x_0\}.$ 
\end{Proposition}
\begin{proof}
We start by noticing that the domain of such a graph cannot contain any curve in the boundary, since this would imply that $\Sigma$ leaves $\mathcal{S}_w,$ while ${\rm width}(\Sigma)=2w.$ In particular, $\Sigma$ is an entire graph over the plane $P_v=[v]^\perp.$

To complete the proof, we only need to check that $w=0.$ Indeed, assume for contradiction that $w>0.$ Let us denote by $\pi$ the plane spanned by $\ee_3$ and $v$. Hence $\pi$ is orthogonal to $P_v$ and transverse to the slab.

By Remark \ref{re:pre}, we obtain two curves $\gamma_1$ and $\gamma_2$ in $\Sigma\cap\pi$ asymptotic to the line $L_i:=\pi\cap\{x=(-1)^iw\}$  at the ``upper'' infinity. Clearly, this means that multiple points on $\Sigma$ project in the direction of $v$ to the same point on $P_v$, a contradiction with $\Sigma$ being a graph in the direction of $v$.
\end{proof}

\subsection{Rad\'o functions}\label{morse-section}

In this subsection, we present some basic facts about  the Morse-Rad\'o theory 
of minimal surfaces in $3$-manifolds (see \cite{HMW23}.)  These facts are essential tools for the rest of the paper.
{
\begin{Definition}\label{rado-def}
A continuous, real-valued function on a $2$-manifold $M$ is called a {\bf Rad\'o function}
provided each point $p\in M$ has a neighborhood $U$ such that
\begin{enumerate}[\upshape(1)]
\item $U\cap \{F=F(p)\}$ consists of a finite collection $C_1,\dots, C_v$ of embedded arcs.
\item Each $C_i$ joins the point $p$ to a point in $\partial U$.
\item $C_i\cap C_j=\{p\}$ for $i\ne j$.
\item\label{crossing-item} Each $C_i$ is in the closure of $\{F>F(p)\}$ and in the closure of $\{F<F(p)\}$.
\item If $p\in \partial M$, we also require that each $C_i\setminus\{p\}$ is contained in the interior of $M$.
\end{enumerate}
The number $v=v(F,p)$ is called the {\bf valence} of $p$.
\end{Definition}

The following result shows that Rad\'o functions arise in  a natural way in the study of minimal surface theory:
\begin{Theorem}\label{structure-theorem}
Suppose $M$ is an embedded minimal surface in a smooth Riemannian $3$-manifold $N$.
Suppose $F:N\to \R$ is a continuous function such that 
\begin{enumerate}[\upshape (1)]
\item The level sets of $F$ are smooth minimal surfaces.
\item Each level set $M[t]:=F^{-1}(t)$ is in the closure of $\{F>t\}$ and of $\{F<t\}$.\label{crossing-condition}
\end{enumerate}
Suppose also that $F$ is not constant on any connected component of $M$.
Then the restriction of $F$ to the interior of $M$ is a Rad\'o function without any interior local maxima or interior local minima.
The interior saddles of multiplicity $n$ are the points where $M$ makes contact of order $n$ with 
the level set $\{F=F(p)\}$.
\end{Theorem}

The main result in \cite{HMW23} is the following:
\begin{Theorem}[Hoffman, Mart\'in, White \cite{HMW23}] \label{th-one}
Suppose that $M$ is a compact minimal surface with boundary in a Riemannian manifold $N$.
Suppose that $F:N\to\RR$ is a continuous function  such that
\begin{enumerate}[\upshape(1)]
\item\label{hypothesis-1} if $\dim N=3$, the level sets of $F$ are minimal surfaces, and 
\item\label{hypothesis-2} if $\dim N>3$, the level sets of $F$ are totally geodesic.
\item\label{hypothesis-3} for each $t$, $\{F=t\}$ is in the closure of $\{F>t\}$ and of $\{F<t\}$.
\end{enumerate}
Suppose also that $F$ is nonconstant on each connected component of $M$,
and that the set $Q$ of local minima of $F|_{\partial M}$ is finite.
Then the number $\mathsf{N}(F|_{M})$ of interior critical points of $F|_{M}$ (counting multiplicity) 
and the number $s^\partial(F)$ of boundary saddle points of $F|_{M}$ (counting multiplicity) satisfy
\begin{equation*}
  \mathsf{N}(F|_{M}) + s^\partial(F) =  |Q| - \chi(M),
\end{equation*}
where 
$\chi(M)$ is the Euler characteristic of $M$ 
and where $|Q|$ is the number of elements in the set $Q$.
\end{Theorem}
In Theorem~\ref{th-one}, ``interior critical point of $F|_{M}$'' means ``interior point $p$ of tangency of $M$ and the level set $\{F=F(p)\}$'',
and the multiplicity of such a critical point is the order of contact of $M$ and $\{F=F(p)\}$.
Boundary saddle points and their multiplicities are defined in \cite[Definition~23]{HMW23}

It would be natural in Theorem~\ref{th-one} to assume that $F$ is $C^1$ (or even smooth) with nowhere vanishing gradient.
However, that assumption would be undesirable for the following reason.
Consider a minimal foliation $\Ff$ of a Riemannian $3$-manifold.  Of course the leaves
are smooth.   At least locally, the foliation can be given as the level sets of a function $F$.
However, for some minimal foliations, there is no such function that is $C^1$ with nowhere vanishing gradient.
(Consider, for example, the minimal foliation of $\{(x,y,z): x>0\}$ consisting of the halfplanes $z= sx$ with $s\ge 0$
and the halfplanes $z=s$ with $s<0$.  If $F$ is a $C^1$ function whose level sets are the leaves, then the gradient $DF(x,0)\equiv 0$.)

In Sections \ref{sec:chini} and \ref{pitchforks-section}, we will show how powerful this type of results are: using just continuous functions.

Theorem~\ref{th-one} provides an exact formula for $\mathsf{N}(F|_{M})$.
In many situations, a good upper bound for $\mathsf{N}(F|_{M})$ suffices.
Simply dropping the term $s^\partial(F)$ in Theorem~\ref{th-one} gives the bound 
\[
   \mathsf{N}(F|_{M}) \le |Q| - \chi(M),
\]
which is often adequate. 
But one can get a better upper bound as follows.
Let $A$ be the set of local maxima and local minima of $F|_{\partial M}$ that are not
local maxima or local minima of $F|_{M}$.  Then $s^\partial(F) \ge |A|$ (where $|A|$ is the number of elements of $A$),
so from Theorem~\ref{th-one}, one can deduce

\begin{Corollary}[Hoffman, Mart\'in, White \cite{HMW23}]\label{th-one-corollary}
Under the hypotheses of Theorem~\ref{th-one}, 
\begin{equation*}
   \mathsf{N}(F|_{M})  \le |Q| - \chi(M) - |A|.
\end{equation*}
\end{Corollary}

See Theorem~26 in \cite{HMW23}, where the authors also specify when equality holds in Corollary~\ref{th-one-corollary}.

\begin{Remark} \label{re:8}
 One sometimes encounters $F$ and $M$ that satisfy all but one of the hypotheses of Theorem~\ref{th-one}, namely the hypothesis that the set of local minima of $F|_{\partial M}$ is finite. 
 In particular, that hypothesis will fail if $F$ is constant on one or more arcs of $\partial M$. One can handle such examples as follows. Suppose $F$ is not constant on any connected component of $\partial M$ (like in Lemma \ref{four_components}.)
  Let $\tilde M$ be obtained from $M$ by identifying each arc of $\partial M$ on which $F$ is constant to a point. 
  Let $\tilde F$ be the function on $\tilde M$ corresponding to $F$  on $M$. If $\tilde F|_{\partial \tilde M}$ has a finite 
   set $\tilde Q$ of local minima, then
\begin{align*}
\mathsf{N}(F|_{M})
&=
\mathsf{N}(\tilde F|_{ \tilde M})  \\
&=  |\tilde Q| -\chi(M) - s^\partial (\tilde F) \\
&\le  |\tilde Q| - \chi(M) -  |\tilde A|,
\end{align*}
where $\tilde A$ is the set of local minima and local maxima of $\tilde F|_{ \partial \tilde M}$ that are not local minima or local maxima     
      of $\tilde F|_{M}$. 
  These facts are direct consequences of Theorems~24, 26, 
and~48 in \cite{HMW23}.
\end{Remark}

There is also a version of Theorem~\ref{th-one} for noncompact $M$:
\begin{Theorem}[Hoffman, Mart\'in, White \cite{HMW23}]\label{noncompact-intro-theorem}
Let $-\infty\le a < b \le \infty$.
In Theorem~\ref{th-one}, suppose the hypothesis that $M$ is compact is replaced by the 
hypotheses that $F:M\to (a,b)$ is proper,
 that $d_1(M):=\dim H_1(M;\ZZ_2)$ is finite, and that
the limit
\[
  \beta:=\lim_{t\to a, \, t>a} | (\partial M)\cap F^{-1}(t)| 
\]
exists and is finite.   
Then
\begin{equation*}
  \mathsf{N}(F|_{M}) + s^\partial(F) = \frac 12\beta+ |Q| - \chi(M),
\end{equation*}
and therefore
\begin{equation*}
  \mathsf{N}(F|_{M}) \le  \frac 12 \beta+ |Q| - \chi(M) - |A|.
\end{equation*}

\end{Theorem}

\begin{Remark} \label{re:semi}
   Another useful fact about $\mathsf{N}(F|_{M})$ is that it depends lower semicontinuously on $F$ and on $M$ (even without assuming 
properness); see Theorem~40 in \cite{HMW23}. 
\end{Remark}

}

A simple application of these results is the following proposition.
\begin{Proposition} \label{prop:vertical}
Let $\Sigma$ be a complete, embedded translating soliton  which is contained in a slab:
$$\{q \in \R^3 \; : \: a < \langle q,\xi \rangle < b\}, \quad \xi\in\mathbb{R}^3, \quad \|\xi\|= 1,$$
where $-\infty<a<b< \infty.$
Then the slab normal is horizontal: $\xi\perp \ee_3$.
\end{Proposition}
\begin{proof} Let $f:\R^2\to \R$ be the function whose vertical graph gives the bowl soliton, $\mathcal{B}$, with $f(0,0)=0.$ Consider the function $\widetilde F: \mathbb{R}^3 \rightarrow \R$
$$\widetilde F(x,y,z)=z-f(x,y),$$
and define
$F:= \widetilde F|_{\Sigma}$. Theorem \ref{structure-theorem} says that $F$ is a Rad\'o function (see Definition \ref{rado-def}.)
If we assume that the slab normal $\xi$, is not perpendicular to $\ee_3$, then there are $C_1,C_2>0$ such that:
\[
\forall (x,y,z)\in\Sigma: |z|\leq C_1 \|(x,y)\| + C_2
\]
This gives $\lim_{\| p\| \to \infty} F(p)=-\infty$, seeing as (by \cite{AW94}, \cite{CSS07} or \cite{Mol14})
\[ f(x,y)=\frac 14 \|(x,y)\|^2-\log (\|(x,y)\|) + O\left(\|(x,y)\|^{-1}\right).\]
Therefore, $F$ has a global maximum at some $p_0=(x_0,y_0,z_0)\in\Sigma$.
But this is impossible, since Theorem \ref{structure-theorem} implies that a Rad\'o function has no local maxima or minima unless it is constant. This contradiction proves the proposition. 
\end{proof}

\section{The wings of translating solitons}\label{Wing}
In this section, we are going to define what we mean by wings of a complete, embedded translating soliton of width $2w>0$, finite genus and finite entropy (see Section \ref{sec:notation} for the notation).

\begin{Proposition}\label{General-control}
Let $\Sigma$ be a complete, connected embedded translating soliton of finite width, finite genus and finite entropy. Let $\{\lambda_n\}$ and $\{\sigma_n\}$ be sequences of real numbers. Consider the sequence $\{\Sigma_n=\Sigma-\lambda_n\ee_2-\sigma_n\ee_3\}.$ If  $\Sigma_n\to\Sigma_\infty,$ then the number of connected components of $\Sigma_\infty$ is finite. Moreover, if $\Sigma$ is a graph, then the multiplicity of convergence on each connected component of $\Sigma_\infty$ is one.
\end{Proposition}
\begin{proof}

The first part is a consequence of the finiteness of the entropy since the entropy is lower semicontinuous.

Hence, what remains to be proven is multiplicity one of the convergence in the graphical case. For this, we fix $p\in\Sigma_\infty$ and $r>0$ small enough so that the open ball $B_r(p)$ in $\R^3$ only intersects one connected component of $\Sigma_\infty$, which we denote $\Sigma'$, and so that $\p B_r(p)$ is transverse to $\Sigma'$. Let $m$ be the multiplicity of $\Sigma'$. We have \[\lim_{n\to+\infty} \mathcal{H}_g^2(\Sigma_n\cap B_r(p))=m\mathcal{H}_g^2(B_r(p)\cap\Sigma_\infty).\] 

Recall that $\Sigma={\rm Graph}_{(v_\theta,\Omega_\theta)}[u_\theta],$ where $u_\theta:\Omega_\theta\to\R$ is a smooth function. In particular, for all $n$ we have $\Sigma_n={\rm Graph}_{(v_n,\Omega_n)}[u_n],$ where $u_n:\Omega_n\to\R$ is a smooth function. We remark that there is $n_0$ large enough satisfying that $\p B_r(p)$ is transverse to $\Sigma_n,$ when $n\geq n_0.$ Using again that the multiplicity is $m$, we also may suppose that $\Sigma_n\cap B_r(p)$ has $m$ disjoint connected components, if $n\geq n_0.$ Choose two disjoint connected components and called they by $C_n^1$ and $C_n^2,$ clearly $\p C_n^i\subset \p B_r(p).$ If $m\geq 2,$ then it may also be assumed that 
\[
\label{main-area-estimate}
\mathcal{H}_g^2(C_n^1)+\mathcal{H}_g^2(C_n^2)\geq\frac{3}{2}\mathcal{H}_g^2(B_r(p)\cap\Sigma_\infty).
\]

\[
\label{main-cylinder-estimate}
\mathcal{H}_g^2(\widehat{\mathcal{C}}_n)\leq\frac{1}{2}\mathcal{H}_g^2(B_r(p)\cap\Sigma_\infty),
\]
and
\[
{\rm dist}_{H}(\p C^1_n,\p C^2_n)<\pi
\]
when $n\geq n_0$, where $\widehat{\mathcal{C}}_n$ is the annulus on $\p B_r(p)$ with boundary $\p C_n^1\cup\p C_n^2$ and ${\rm dist}_{H}(\cdot,\cdot)$ indicates the Hausdorff distance in $\R^3$ with respect to the Euclidean metric.

For all $n\geq n_0$ let $\mathcal{C}_n$ be the $g-$area minimizing surface in $\R^3$ with boundary $\p C_n^1\cup\p C_n^2$. The existence of such a surface is ensured by Theorem \ref{Least-area}. Note also that the proof of Proposition \ref{Proper-condition} implies that $C_n$ is connected and lies in $\Omega_n\times_{v_n}\R.$

Now, since \[\mathcal{H}_g^2(\mathcal{C}_n)\leq\mathcal{H}_g^2(\widehat{\mathcal{C}}_n)<\mathcal{H}_g^2(\Sigma_n\cap B_r(p)),\]when $n\geq n_0$ which leads to a contradiction with Lemma \ref{Area-minimizing}. Therefore, we must have $m=1$, which completes the proof. 
\end{proof}

We are now ready to prove one of the main results of this paper, which shows that there is a well-defined notion of ``the number of wings'' for  translating solitons of finite width, finite genus and finite entropy. This notion of wing will be fundamental to our further analysis.
Before stating the theorem, we record a lemma which is needed for its proof.

\begin{Lemma} \label{cylinder}
Let $W$ be a connected, embedded translator with (possibly non-compact) boundary contained in $S_w \cap \{y \geq t\}$, for some $t>0$ and so that $\partial W \subset \pi_y(t)$. If $$\sup \{y \; : \; (x,y,z) \in W\} < \infty,$$ then $W$ is a piece of the plane $\pi_y(t).$
\end{Lemma}
\begin{proof}
The proof of this lemma is a slight modification of the arguments in the proof of the bi-halfspace theorem (wedge theorem) due to Chini and M\o{}ller (see Theorem 1 in \cite{CM21}).
\end{proof}

The following is our main result concerning the wings of translating solitons in slabs. For the statement, recall our notation from p. \pageref{sec:notation}.

\begin{Theorem}\label{Wings-number}
Let $\Sigma$ be a complete, connected embedded translating soliton of finite width, finite genus and finite entropy. There exist positive integers $\omega^{\pm}(\Sigma)$, and a $t_0>0$ such that if $ t\geq t_0,$ then the number of connected components of $\Sigma^{+} (t)$ (respectively $\Sigma^{-} (-t)$) equals $\omega^{+}(\Sigma)$ (respectively $\omega^{-}(\Sigma)$).
\end{Theorem}

\begin{proof}
First, let us prove that $\omega^+(\Sigma)$ is well-defined. 
The proof for $\omega^-(\Sigma)$ is similar.
Let us define the non-negative integers $$\mu(t):= \mbox{number of connected components of $\Sigma^+(t)$}.$$

\begin{Claim} \label{cl:uno}
The function $\mu(t)$ is bounded.
\end{Claim}
Assume that this were not the case. Then it would be possible to find a sequence $t_n\nearrow+\infty$ so that $\Sigma^+(t_n)$ has more than $n$ connected components.
Notice that (by the maximum principle) none of these connected components can be bounded from above.
Let $C_j(t_n)$ be these components, for $j=1, \ldots, k(n),$ and define
$$\beta_j(t_n):=\inf\{z \;: \; z \in C_j(t_n)\} \in [-\infty,+\infty). $$
Finally, take
$$ \alpha_n=\max\{n, \beta_j(t_n), j=1, \ldots, k(n)\}. $$
Then, the subsequential limit 
$$\Sigma_\infty:= \lim_m \left(\Sigma- t_n \ee_2-\alpha_n \ee_3\right),$$
exists (Theorem \ref{theorem-compactness}), is not empty and has infinitely many connected components, contradicting Proposition \ref{General-control}.
\begin{Claim} \label{cl:dos}
The function $\mu(t)$ is a non-decreasing function.
\end{Claim}
Again, we proceed by contradiction. Assume that we have $t<t'$ but $\mu(t)>\mu(t')$. This means that either:
\begin{enumerate}[(a)]
    \item there are two components $W_1 \neq W_2$ in $\Sigma^+(t)$ that form parts of the same component $W$ of $\Sigma^+(t')$, or
    \item there exists a connected component $W$ in $\Sigma^+(t)$ whose intersection with $\Sigma^+(t')$ is empty.
\end{enumerate}
Statement (a) is impossible because $\Sigma^+(t') \subset \Sigma^+(t).$ On the other hand, the statement (b) would imply (using Lemma \ref{cylinder}) that $W$ would be planar. So, the plane $\pi_y(t)$ would be contained in $\Sigma$, which is absurd. This contradiction proves Claim \ref{cl:dos}.

Taking into account Claims \ref{cl:uno} and \ref{cl:dos}, there must be a $t_0>0$ such that the function $\mu(t)$ is constant for $t>t_0$. Then $\omega^+(\Sigma):=\mu(t)$ ($t\geq t_0$) is well-defined.

Finally, if $\omega^-(\Sigma)=0$, this would imply that $\Sigma$ were contained in a region of the form $(-w,w) \times (a,\infty) \times \R$. But this is impossible due to  Theorem 1 in \cite{CM21}. Thus we always have that $\omega^{-}(\Sigma)\geq1$ and similarly $\omega^{+}(\Sigma)\geq1$.
\end{proof}

\begin{Definition}[\bf Wings of a translating soliton]
Let $\Sigma$ be a complete, connected translating soliton contained in a slab $\{-w\leq x \leq w\}$ of finite genus and finite entropy, and let $t_0>0$ be a real number as in Theorem \ref{Wings-number}. Then, by a {\it wing} of $\Sigma$ we mean any connected component of $\Sigma \setminus \{|y| \leq t\}$, where $t>t_0$.
\end{Definition}

\begin{Definition}
The number $\omega^+(\Sigma)$ (respectively $\omega^-(\Sigma)$) from Theorem \ref{Wings-number} is called the number of right (respectively, left) wings of $\Sigma.$ The number $\omega(\Sigma)=\omega^+(\Sigma)+\omega^-(\Sigma)$ is called the total number of wings of $\Sigma.$
\end{Definition}

Before stating the next result, we need the following lemmas.

{\begin{Lemma}\label{end-finite}
Let $\Sigma$ be as before. Then, the number of ends of $\Sigma$ is bounded from above by $\omega(\Sigma)$.
\end{Lemma}
\begin{proof}
As $\Sigma$ has finite genus, then $\Sigma$ is homeomorphic to a compact surface, denoted by $\overline{\Sigma}$, minus a possibly infinity set. Let $\{E_1,E_2,\ldots,E_n\}$ be any finite sequence of ends of $\Sigma.$ Adopt on $\overline{\Sigma}$ any Riemannian metric and define
\(4d=\min\{{\rm dist}_{\overline{\Sigma}}\{E_i,E_j\}\;:\;i\neq j\}.\) With these choices, let $\{D_{d}(E_i)\}$ be a family of disjoint geodesic disks on $\overline{\Sigma},$ and thus, $\{D_d(E_i)\setminus\{E_i\}\}$ is a disjoint family of open set on $\Sigma$, here and for the remaining of the proof we will see each $D_d(E_i)\setminus\{E_i\}$ as a subset of $\Sigma$ rather than $\overline{\Sigma}.$ 

Now, Theorem 1 in \cite{CM21} says that each $D_d(E_i)\setminus\{E_i\}$ cannot belong to $\{|y|<R\},$ for all $R>0.$ Moreover, we can find $R_1$ large enough so that $\p D_d(E_i)\setminus\{E_i\}\subseteq\{|y|\leq R_1\}.$ Consequently, if $t>R_1,$ then $\Sigma^\pm(\pm t)$ possesses $n$ disjoint connected components, but then Theorem \ref{Wings-number} implies that
\(n\leq\omega(\Sigma).\) This finishes the proof.
\end{proof}}

\begin{Lemma}\label{Lemma-control}
Consider $t_0>0$ given by Theorem \ref{Wings-number}. There is $t_1\geq t_0$ so that:
\begin{enumerate}
    \item {$\Sigma \cap\{|y|\geq t_1\}$ is a disjoint union of simply connected regions of $\Sigma$}, and
    \item if $t>t_1$ then $\Sigma^-(t)$ and $\Sigma^+(-t)$ are connected, and
    \item {if $t>t_1$ then $\Sigma^-(t)$ and $\Sigma^+(-t)$ are homeomorphic to $\Sigma$.}
    
\end{enumerate}
\end{Lemma} 
\begin{proof} Firstly, we are going to prove statement {\em (1)}.
{As $\Sigma$ has {finite topology, by Lemma \ref{end-finite}}, then $\Sigma$ is homeomorphic to a compact surface $\overline{\Sigma}$ minus a finite set of points
$\{E_1, \ldots,E_k\}.$}
{If $t>t_0$ is large enough, then $$M_t:=\Sigma \cap \{|y|\geq t\}$$ is contained in $D_1 \cup \cdots \cup D_k$, where $D_i$ is a topological disk in $\overline{\Sigma}$ centered
at the end $E_i$, $i=1, \ldots k.$}

{Label $D_i(t):= M_t \cap D_i.$ We have two possibilities \begin{enumerate}[(a)]
    \item $E_i \in {\rm Int}(D_i(t))$, i.e. $\partial D_i(t)$ is compact in $\Sigma \subset \R^3.$
    \item $E_i \in \partial (D_i(t))$, i.e. $\partial D_i(t)$ is not compact in $\Sigma \subset \R^3.$
\end{enumerate}}
{Notice that $D_i(t') \subset D_i(t)$, for $t'>t$. Thus,
if (b) happens for some $t$, then it happens for all $t'>t.$ Notice also that in Case (b) $D_i(t)$ is simply connected.}

{Hence, it suffices to prove that (a) cannot happen for all $t>t_0$. Assume that this were the case, then we could take $t'>t$ such that $t'-t>\pi.$ Then $D_i(t) \setminus D_i(t')$ would be a compact translator in $\R^3$ whose boundary curves are contained in $\pi_y(y)$ and $\pi_y(t').$ At this point we consider $$\Psi: \{t <y < t'\} \longrightarrow \R$$
$$ \Psi(x,y,z)=z+\log(\sin(y-t)).$$}
{$\Psi|_{D_i(t) \setminus D_i(t')}$ is a Rad\'o function and it has a local minimum in the interior of $D_i(t) \setminus D_i(t')$, which is absurd (we use again Proposition 2.3 in \cite{HMW23}.) This contradiction proves that (a) cannot happen for all $t>t_0$. So, Statement {\em (1)} has been proved. }

In order to prove {\em (2)}, we proceed by contradiction.
Assume it is false, then there would be an increasing sequence $\{t_n\}$ so that, for instance, $\Sigma^-(t_n)$ is disconnected for all $n$. We start with $\Sigma^-(t_1),$ let $A_1$ and $B_1$ be two disjoint connected components of it. Namely, the existence of these subsets ensures that $\omega^-(\Sigma)\geq2.$ 

On the other hand, as $\Sigma$ is connected, there is $t_{n_1}$, $n_1>1$, so that $A_1$ and $B_1$ are contained in the same connected component of $\Sigma^-(t_{n_1}).$ Let us call such a connected component $A_{n_1}$, and let $B_{n_1}$ be a distinct connected component of $\Sigma^-(t_{n_1}).$ So $\omega^-(\Sigma)\geq3.$ 

Repeating the argument inductively, we would conclude that $\omega^-(\Sigma)\geq k$ for all $k\in\mathbb{N}$, which is contrary to Theorem \ref{Wings-number}.

{Finally, in order to prove {\em (3)}, notice that from {\em (1)} and {\em (2)} we have that $\Sigma^-(t)$ is homeomorphic to $\overline{\Sigma}$ minus a finite number of topological disks which contained one end $E_i$ (for some $i$) at the boundary.
Hence, $\Sigma^-(t)$ has the same topology as $\Sigma=\overline{\Sigma} \setminus \{E_1, \ldots,E_k\}.$}

\end{proof}

Next, we show that there exists a relation between the number $\omega^{\pm}(\Sigma)$ and the number of connected components of $\p\Sigma^{\pm}(t)$ whenever $t> t_1.$ 
\begin{Proposition}\label{boundary-control}
If $t>t_1,$ then the number of connected components of every slice $\p\Sigma^{\pm}(\pm t) = \Sigma\cap \pi_y(\pm t)$ is exactly equal to $\omega^{\pm}(\Sigma).$
\end{Proposition}
\begin{proof}
{Let $W$ be any connected component of  $\Sigma^{+}(t),$ with $t>t_1.$ Assume for contradiction that the number of connected components of $\p W$ is at least $2$. Let $\gamma_1$ and $\gamma_2$ be two disjoint connected components of $\p W$ and fix two points $p_i\in\gamma_i.$  Since  W is connected, there must exist a path $\sigma_W$ on $W$ connecting $p_1$ and $p_2.$ On the other hand, as $\Sigma^-(t)$ is connected, there is a path $\sigma$ on $\Sigma^{-}(t)$ connecting $p_1$ and $p_2$. However, the loop $\sigma_W*\sigma^{-1}$, which we get from concatenation, is not homotopically equivalent to a curve in $\Sigma^-(t)$ ({\em (3)} in Lemma \ref{Lemma-control}). This would imply that $p_1$ and $p_2$ could be joined by a continuous arc contained in $\Sigma \cap \pi_y(t),$ which is absurd. This contradiction proves the proposition.}
\end{proof}

\begin{Corollary}\label{transverse-result}
If $W$ is any connected component of $\Sigma^{\pm}(\pm t),$ with $t> t_1$, then $W\cap \pi_y(\pm t)$ is a smooth, connected proper curve.  In particular, if $t>t_1$ then $\pi_y(-t)$ and $\pi_y(t)$ are transverse to $\Sigma.$
\end{Corollary}
\begin{proof}
Proposition \ref{boundary-control} implies that $\partial W$ is connected. 

If we consider the function $f_{\ee_2}:\Sigma \rightarrow \R$
\[f_{\ee_2}(x):= \langle x,\ee_2\rangle,\]
then Theorem 2.2 in \cite{HMW23} says that $f_{\ee_2}$ is a Rad\'o type function. Hence, if the set $\p W \subseteq \{f_{\ee_2}=t\}$ contains a non-smooth curve it is because $\partial W$ contains a ramification point of $f_{\ee_2}$ where, in suitable polar coordinates for $M$ at $p$, $f_{\ee_2}$ has the form:
\[ f_{\ee_2}(p)+c \, r^n \sin n \, \theta+ o(r^n),\]
for some integer $n \geq 2$ and $c\neq 0$ (see Definition \ref{rado-def}).

So, for $t' > t$ and very close to $t$, since $\{f_{\ee_2}=t'\}$ never contains compact components (by Theorem 2 in \cite{CM21} or Corollary 2.4 in \cite{HMW23}), the set $\{f_{\ee_2}=t'\}\cap W$ is a curve with $n$ different components. This contradicts the connectedness of $\p \left(W \cap \Sigma^+(t')\right)$.
\end{proof}

Later we will also need the case of vertical slicing non-perpendicularly to the slab, which we record in the following corollary.

\begin{Corollary}\label{boundary-control-skew}
Suppose $t>t_1,$ and $P$ is any vertical plane transversal to the slab $S_w$ such that
$P\cap S_w \subseteq \{\pm y > t_1\}$. Then
the number of connected components of $P\cap \Sigma$ equals $\omega^\pm(\Sigma)$.
\end{Corollary}

\begin{proof} If $P$ is parallel to $\pi_y(t_1)$, then we have nothing to prove (Proposition \ref{boundary-control}.) So we are going to assume that $P$ is not parallel to $\pi_y(t_1).$

We consider the case $y > t_1$. The proof of the case $y<-t_1$ is symmetric.
We have that $$P\cap \partial \mathcal{S}_w= \{x=w, \; y=t_0\} \cup 
\{x=-w, \; y=t_2\}, $$
with $t_1 \leq t_0, \, t_2$. As we are assuming that $P$ is not perpendicular 
to $\ee_2$, then $t_0 \neq t_2.$ W.l.o.g. we can assume that $t_0<t_2$ (the other case is similar.)

Consider $S_P$ the connected component of $\mathcal{S}_w \setminus P$ which contains $\Sigma^+(t_2)$ and define 
\[ W_P:= \Sigma \cap S_P\]

Namely note that if $T$ denotes the closed triangular cylinder in $S_w$ between $P$ and $\pi_y(t_2)$, then 
$$W_P = \Sigma^+(t_2)\cup (T\cap \Sigma)$$
with $ \Sigma^+(t_2)\cap (T\cap \Sigma) \neq \varnothing$, and we find that $T\cap \Sigma$ has the same number of connected components as $\Sigma^+(t_2)$. Actually, each connected component of $W_P$ is contained in the corresponding connected component of $\Sigma^+(t_2)$: If not, $T\cap \Sigma$ would contain a connected component which does not intersect $\pi_y(t_2)$, seeing as $\pi_y(t_2)\cap \Sigma$ and $\pi_y(t_0)\cap \Sigma$ have the same number of connected components by Proposition \ref{boundary-control}. But such a component thus has boundary contained in $P$, which is forbidden by Lemma \ref{cylinder} unless $\Sigma = P$, which is absurd.

Finally, knowing now that $W_P$ is connected, we can repeat the proof of Proposition \ref{boundary-control}, using that  $\Sigma$ has finite topology, to conclude that the number of connected components of $W_P$ agrees with $\omega^+(\Sigma)$.
\end{proof}

\section{Structure of the set $N(\{H=0\})$} \label{H-section}

In this section, we will study some finer properties of the zero level set of the mean curvature function on translators and the images of this set under the Gauss map. This will be crucial in order to conclude certain mean convexity results for (limits) of sequences of complete translators, and in turn allow us to make use of known classification results for such solitons.

For this purpose, we will from now on label
\begin{equation}\label{H-def}
\mathcal{H}=\{p \in \Sigma \: : \; H(p)=0\}= N^{-1}(E),
\end{equation}
where we use the following notation for the equator in the unit sphere (i.e. the codomain of the Gauss map):
\begin{equation}
E=\left\{\nu \in \SS^2 \: : \; \langle \nu, \ee_3\rangle =0\right\}.
\end{equation}

When a translator $\Sigma$ is not a vertical plane, then the set $\mathcal{H}$ is the union of a locally finite collection of regular $C^2$-curves for which the singular points, defined as the intersections of the curves, are isolated (see Remark 18 in \cite{Chi20}). Note in particular that $\mathcal{H}$ cannot contain isolated points.

Contained within the proof of Proposition 20 in \cite{Chi20} is the following lemma, guaranteeing in particular the absence of closed loops in $\mathcal{H}$ in the simply connected case, which by Lemma \ref{domain} includes the case of graphs.

\begin{Lemma}\label{pi_one_injection}
Let $\Sigma^2\subseteq\R^3$ be a complete embedded translator.
Then there is an injection of fundamental groups,
\[
\pi_1\left(\mathcal{H}\right)\hookrightarrow \pi_1(\Sigma).
\]
In particular, the zero sets of the mean curvature on complete simply connected translating solitons cannot have connected components which are compact.
\end{Lemma}

\begin{Lemma}\label{H-open}
Let $\Sigma$ be a nonplanar connected properly embedded translator. Then $N(\mathcal{H})$ is open subset of the equator $E$.
\end{Lemma}
\begin{proof}
If the Gauss map is a diffeomorphism near $p_0$, this means that $p_0$ is not one of the singular points in $\mathcal{H}$, and by the inverse function theorem applied to the Weingarten map $dN_{p_0}$, this is equivalent to $k_1(p_0) = -k_2(p_0) \neq 0$. In this case, it is clear that open arcs in $\mathcal{H}$ are mapped to open arcs in $N(\mathcal{H})$.

In the case when $dN_{p_0}$ is not an isomorphism, we have $k_1(p_0) = 0 = k_2(p_0)$. Letting $v_0:= N(p_0)$ and making use of the Rad\'o{} function $f_{v_0}(p):=\langle v_0, p\rangle$ (see again Theorem 2.2 in \cite{HMW23}), for small enough $\varepsilon>0$ the surface $\Sigma\cap B(p_0,\varepsilon)$ can be written as a graph over the plane $T_{p_0}\Sigma = \{\langle v_0, p\rangle = 0\}$ of the following form (in suitably rotated polar coordinates on $T_{p_0}\Sigma$):
\begin{equation}\label{Rado-graph}
f_{v_0}(r,\theta) = c_0 + c_nr^n\sin(n\theta) + o(r^n), \quad c_n \neq 0,
\end{equation}
where $n\geq 2$ (since $k_1(p_0) = 0 = k_2(p_0)$). In this approximation, the $n-1$ crossing curves in $\mathcal{H}$ correspond to the straight lines through the origin at angles $\theta = \frac{\pi}{2n} + \frac{k\pi}{n-1}$, $k\in \{0,\ldots, n-2\}$.
Then, a straightforward computation using \eqref{Rado-graph} gives that $N(\mathcal{H}\cap B(p_0,\varepsilon)) $
contains an open arc in $E$ and that $N(p_0)$ is an interior point of this arc.

In fact, when $n$ is odd, each line $\theta=\frac{\pi}{2n} + \frac{k\pi}{n-1}$ is mapped onto an arc in $E$ of the form 
$\{w \in E \; : \; \langle J v_0,w \rangle \in (-\delta, \delta)\}$, for some $\delta>0$ sufficiently small, where $J$ represents the counter-clockwise $90^{\circ}$ rotation in the $(x,y)-$plane. When $n$ is even, each line $\theta=\frac{\pi}{2n} + \frac{k\pi}{n-1}$ is mapped onto an arc in $E$ of the form:
\begin{itemize}
    \item $\{w \in E \; : \; \langle J v_0,w \rangle \in [0, \delta)\}$, when $k$ is even, and
    \item $\{w \in E \; : \; \langle J v_0,w \rangle \in (-\delta, 0]\}$, when $k$ is odd.
\end{itemize}
In both cases, we conclude that $N(p_0)$ is contained in the interior of an arc in $N(\mathcal{H}).$ Note that this property is stable under perturbation, so that the approximation in \eqref{Rado-graph} is sufficient.
\end{proof}

\begin{Lemma} \label{gauss-arc}
Let $\Sigma$ be a connected complete embedded translating soliton of width $2w>0$, finite genus and finite entropy. Then $N(\mathcal{H})$ is an open subset of $E$ whose boundary points (if any) lie in $\{-\ee_1,\ee_1\}.$ In particular, if
$N(\mathcal{H})\cap \{-\ee_2,\ee_2\}=\varnothing$, then $\mathcal{H}=\varnothing.$
\end{Lemma}

\begin{proof} If $\partial N(\mathcal{H})\neq \varnothing$, then let $e\in E$ be a boundary point of $N(\mathcal{H})$. If there were a point $p\in \Sigma$ so that $N(p)=e,$ then by Lemma \ref{H-open}, there would be a small arc in $N(\mathcal{H})$ containing $e$, which is absurd. 

Thus, we know that there must be a divergent sequence $\{p_n\}$ in $\Sigma$, so that
$N(p_n) \to e.$ Take the sequence
$$\Sigma_n:= \Sigma-p_n.$$
By Theorem \ref{theorem-compactness}, we know that (up to taking a subsequence) $\Sigma_n$ has a limit, which we call $\Sigma_\infty$. Let $\Sigma_0$ be the connected component of $\Sigma_\infty$ containing the origin. From Lemma \ref{H-open} we have that if $\Sigma_0$ is not planar, then $N(\Sigma_0)$ is an open subset of $E$ containing $e$. But, as the sequence $\Sigma_n:= \Sigma-p_n$ consists of translates of $\Sigma$, this would imply that $e$ would not be a boundary point of $N(\Sigma)$, which is absurd. Hence, $\Sigma_0$ is a plane perpendicular to $e$. As $\Sigma_0$ is contained in a slab parallel to $\{x=0\}$, then, as claimed: $e=\pm \ee_1.$

The final conclusion of the lemma now follows from the possibilities for $N(\mathcal{H})$, being either one or both of the half-equators in $E\setminus\{-\ee_1,\ee_1\}$ or empty.
\end{proof}
\begin{Lemma} \label{gauss-image-1}
Let $\Sigma$ be a complete embedded translating soliton of width $2w>0$, finite genus and finite entropy. Then for all $\nu\in E\setminus\{-\ee_1,\ee_1\}$:
\[
\#\left(N^{-1}(\{\nu\})\right) < \infty.
\]

\end{Lemma}

\begin{proof}
Fix $\nu \in  E\setminus\{-\ee_1,\ee_1\}$. We consider the Rad\'o{} function
\begin{equation}\label{rado-plane}
f_\nu(p) := \langle p, \nu\rangle,\quad\mathrm{on}\: \Sigma,
\end{equation}
whose critical points are precisely the points of $\Sigma$ where the tangent plane is perpendicular to $\nu$.

As we already reason in the proof of Corollary \ref{transverse-result}, each time that $\{f_\nu=t\}$ contains a critical point of $f_\nu$, the number of connected components of $\{f_\nu=t'\}$ increases, for $t'>t$, when $t>t_1$, with $t_1$ as in Lemma \ref{Lemma-control}. Hence, the number of critical points (counting multiplicity) of $f_\nu$ is finite, because  the number of connected components of $\{f_\nu=t \}$ is a fixed number for all $|t| > t_1$ by Corollary \ref{boundary-control-skew}.
\end{proof}

\section{Weak asymptotic behaviour}

The next step will consist of proving that in the ``upper'' and ``lower'' infinities, complete embedded translating solitons in slabs of finite genus and  finite entropy look like a finite union of planes. Unfortunately, we can only prove a form of weak asymptotic behaviour, as this limit is just subsequential.

\begin{Proposition}\label{weakly-infinity-limit}
Let $\Sigma$ be a complete, embedded translating soliton of width $2w>0$, finite genus and finite entropy. Let $\{\tau_n\}$ be a sequence so that $\tau_n\nearrow+\infty$. Then, after passing to a subsequence, we have
\[
\Sigma\pm\tau_n\ee_3\to\Sigma_{\pm\infty},
\]
where $\Sigma_{\pm\infty}$ is a finite union of parallel planes in $\mathcal{S}_w.$
\end{Proposition}
\begin{proof}
The limit translator exists by Theorem \ref{theorem-compactness}. Consider $\Sigma'$ a non-planar connected component of $\Sigma_{\pm\infty}$, i.e. of finite width. Applying Lemma \ref{gauss-image-1} to $\Sigma$, in the limit the image of the Gauss map 
$$\pm \ee_2 \notin N(\Sigma').$$
Then, by Lemma \ref{gauss-arc}, since $\Sigma'$ is not a plane, we have that $N(\Sigma') \cap E =\varnothing$. Thus $\Sigma'$ is a complete self-translating graph over the $(x,y)$-plane. The next claim will however prove that this cannot happen.
\begin{Claim} \label{jorge}
$\Sigma'$ cannot be a graph over the $(x,y)$-plane.
\end{Claim}
To prove this claim, we proceed by contradiction. By the classification of complete vertical graphs in Theorem \ref{Classification} from \cite{HIMW19}, and since $\Sigma'$ has finite width, we know that $\Sigma'$ is either a tilted grim reaper or a $\Delta$-wing.

Assume first that $\Sigma'$ were a $\Delta$-wing, whose apex, for some $t$, lay in the plane $\pi_y(t)$. Consider the grim reaper given by the 
graph of the function
\begin{eqnarray} 
& f:  \R \times (t-\pi/2,t+\pi/2)  \rightarrow \R, \label{eq:cu}\\
  & f(x,y)  = -\log \cos (y-t), \nonumber
\end{eqnarray}
and let
$$F(x,y,z)=z-f(x,y).$$
Label $M:= \Sigma(t-\pi/2,t+\pi/2), $ (see p. \pageref{sec:notation} for notation.) 
We consider a sequence $z_n \nearrow +\infty$ so that, if we define
$$M_n=M \cap \{|z|<z_n\},$$
then $F|_{\partial M_n}$ has no critical points.
Hence, by Theorem \ref{noncompact-intro-theorem} we have that there exists a constant $C>0$, independent of $z_n$, such that the number of critical points of $F|_{M_n}$ satisfies
$$\mathsf{N}(F|_{M_n}) <C.$$ 
As $M_n$ is an exhaustion of $M$, then, taking Remark \ref{re:semi} into account, we have
\begin{eqnarray} \label{liminf}
\mathsf{N}(F|_{M}) \leq \liminf_n \mathsf{N}(F|_{M_n}) \leq C.
\end{eqnarray}
In particular, there are at most finitely many points in 
$\Sigma \cap \pi_y(t)$ at which the Gauss map is vertical. When we take limits of a sequence of the form $\Sigma_n = \Sigma - \tau_n \ee_3$, where each $\Sigma_n$ has the same Gauss image, this implies that at no points in $\Sigma' \cap\pi_y(t)$ can the Gauss map be vertical. But this contradicts the assumption that $\Sigma'$ is a $\Delta$-wing with vertical Gauss map at some point in the plane $\pi_y(t)$. So this case does not happen.

Assume now that $\Sigma'$ is a tilted grim reaper whose normal vector
along its line of symmetry equals some fixed vector $w_0 \in \SS^2.$ Consider the function $f$ given by \eqref{eq:cu}, for $t=0$ this time, and define $M$ as before. We know that there is $y_0 \in (-\pi/2,\pi/2)$ such that the normal of $\graph[f]$
at $y=y_0$ is $w_0$.  Reasoning as in the case of the $\Delta$-wing, we can conclude that there are at most finitely many points in $\Sigma \cap \pi_y(y_0)$ at which the Gauss map is $\pm w_0.$ Taking limits we again conclude that there were in fact no points in $\Sigma'\cap \pi_y(y_0)$ whose Gauss map equals $w_0$, which is absurd. This proves that also the tilted reaper case cannot happen, proving the claim and hence the proposition.
\end{proof}

\section{Properties of the wings}\label{properties} Let $\Sigma$ be a complete, embedded translating soliton of width $2 w >0$, finite genus and finite entropy.

From now on, we shall assume that $W$ is a wing, i.e. a connected component of $\Sigma^+(t), t>t_1$.
We begin our study by analyzing the projections to the $xy$-plane of such wings.

For the following result, recall that we denote by $\Pi_z$ the orthogonal projection to the plane $\{z=0\}$.

\begin{Proposition}\label{projection}
Letting
\[
\kappa:=\displaystyle\inf_{p\in\partial W}\langle p,\ee_1\rangle, \quad\mathrm{and}\quad \rho:=\displaystyle\sup_{p\in \p W}\langle p,\ee_1\rangle.
\]
Then $\kappa < \rho$ and
\[
\Pi_z(W\setminus\p W)\subseteq(\kappa,\rho)\times(t,+\infty).
\]

\end{Proposition}
\begin{proof}
We will first explain how Chini's argument \cite[Theorem 10]{Chi20} ensures that 
\[
\Pi_z(W)\subset[\kappa,\rho]\times[t,+\infty).
\]

Indeed, normalize so that $\kappa=t=0$ and consider the grim reaper \begin{multline*}
    \mathcal{G}(h)=
    \left\{\left(x,y,-\log\cos\left(y+\frac{\pi}{2}\right)+h\right)\;:\;y\in\left(-\pi,0\right), x\in\R\right\},\\ h\in\R.
\end{multline*}  Denote by $R_\theta$ the counterclockwise rotation of angle $\theta$ around the $z-$axis. By assumption $R_\theta(\mathcal{G}_h)\cap W=\varnothing$ for all $\theta\in[0,\varepsilon]$ for some $\varepsilon>0.$

Let $\theta_0=\sup\{\theta\in\left[0,\frac{\pi}{2}\right]\;:\;R_\theta(\mathcal{G}(h))\cap W=\varnothing\}$. We claim that $\theta_0=\pi/2.$ In fact, let us suppose by contradiction that $\theta_0<\pi/2.$ Since $\Sigma\subset S_w$ and $\p W\subset\pi_y(0),$ then either $R_{\theta_0}(\mathcal{G}(h))\cap W\neq\varnothing$ or $\dist\{R_{\theta_0}(\mathcal{G}(h)),W\}=0.$ 

The first case cannot happen by the maximum principle once $W$ cannot leave $S_w.$ About the second case, we could find a sequence $\{p_n\}\subseteq {\rm int}(W)$ so that $\dist\{p_n,\p W\}\geq \varepsilon$ for some $\varepsilon>0$, $\dist\{p_n,\mathcal{G}(h)\}=0$, and $W$ lies in the non-convex region in $\R^3$ bounded by $\mathcal{G}(h).$ Assume that  $\Sigma-p_n\to\Sigma_\infty$ and $\mathcal{G}(h)-p_n\to\mathcal{G}_\infty.$ Denote by $\Tilde{\Sigma}$ and $\widetilde{\mathcal{G}_\infty}$ the connected components of these sets which contain $(0,0,0).$ Namely, our hypotheses provide that
\[
\Tilde{\Sigma}=\widetilde{\mathcal{G}_\infty},
\]
which is impossible since $\Tilde{\Sigma}$ would be included in a slab $S_{\Tilde{w}},$ where $\Tilde{w}\geq w,$ and $\widetilde{\mathcal{G}_\infty}$ is not included in it.

Therefore, since $h$ was arbitrary, we have proven that $\Pi_z(W)\subset[\kappa,+\infty)\times[t,+\infty)$. Analogously we see that $\Pi_z(W)\subset(-\infty,\rho]\times[t,+\infty).$ This concludes the proof of the claim. We also notice that the maximum principle guarantees that in fact
\[
\Pi_z(W\setminus\p W)\subset(\kappa,\rho)\times(t,+\infty).
\]
\end{proof}

\begin{Proposition}\label{Wing-1}
Let $W_\infty$ be any limit of a convergent subsequence of $\{W-y_n\ee_2-z_n\ee_3\}$, where $\{y_n\}$ and $\{z_n\}$ are sequences with $y_n\nearrow+\infty$. Then each connected component of $W_\infty$ is either a plane or a complete graph over the $(x,y)$-plane.
\end{Proposition}
 
\begin{proof}
Using Lemma \ref{gauss-image-1} we know that:
$$\#(N^{-1}(\{-\ee_2,\ee_2\} \cap W) <C$$
for some $C>0$. Then, if we consider  the subsequential limit $W_\infty$ (which exists by Theorem \ref{theorem-compactness}), then 
$$\pm \ee_2 \notin N(W_\infty).$$
Hence, we can apply Lemma \ref{gauss-arc} to deduce that any connected component of $W_\infty$ must be either a plane or a vertical graph.
\end{proof}

Let $W$ be a connected component of $\Sigma^+(t_1)$. By Corollary \ref{transverse-result} we know that
\(
W\cap \pi_y(s) 
\) is a proper embedded curve on $\pi_y(s)\cap\mathcal{S}_w$, for all $s>t_1.$ Note also that the transversality in Corollary \ref{transverse-result} ensures that the property of a slice curve $W\cap \pi_y(s)$ being bounded from below (or not) in the $\ee_3$-direction is a condition which is open in the interval of $s$-values, and hence this is independent of $s$.

Evidently, the same transversality argument also means that both ends of $W\cap \pi_y(s)$ cannot go to $-\infty$, i.e. $W\cap \pi_y(s)$ be bounded from above, since then it would hold for all values of $s$ and then a two-sided maximum principle argument using a grim reaper cylinder would lead to a contradiction. Thus, either both ends of $W\cap \pi_y(s)$ go to $+\infty$ or one of them goes to $+\infty$ while the other one goes to $-\infty$, where $\pm\infty$ again means with respect to the $z-$axis.

All in all, the following notion of ``wing type'' is well-defined.

\begin{Definition}\label{def:wing-type}
If both ends of the slice curve $W\cap \pi_y(s)$, $s>t_1$, go to $+\infty$ with respect to the $z-$axis, then we say that $W$ is a wing of grim reaper type. On the other hand, if only one end goes to $+\infty$ with respect to the $z-$axis, we say that $W$ is a wing of planar type.
\end{Definition}

The next result summarizes the results that we are going to need about planar type wings.

\begin{Proposition}\label{planar-property}
Let $W$ be a right planar type wing of $\Sigma$ so that $\Pi_z(W\setminus\p W)\subseteq(\kappa,\rho)\times(t,+\infty)$. If $s$ is sufficiently large, then $W^+(s)$ is a graph into the direction of $\ee_1$.
\end{Proposition}
\begin{proof}

First of all, we are going to prove the following:\\
\noindent{}\textbf{Claim:} There exists $s>0$ so that
$\langle N(p),\ee_1 \rangle \neq 0$, for all $p\in W^+(s)$.

Assume this claim were false. Then there would exist sequences $t_n\nearrow+\infty$
and $p_n \in W \cap\pi_y(t_n)$ with $N(p_n) \perp \ee_1.$
If we consider $W-p_n$, then Theorem \ref{theorem-compactness} allows us to take a subsequential limit, which we call $W_\infty$, a complete translator of finite width. Let $W'$
be the connected component of $W_\infty$ containing the origin $0$. Now, Proposition \ref{Wing-1} ensures that $W'$ is either a vertical plane or a graph over the $(x,y)$-plane. But since the Gauss map $N(0) \perp \ee_1$, the planar case is ruled out. Knowing therefore that $W'$ is a complete graph over the $(x,y)$-plane, of finite width, by Theorem \ref{Classification} it must be either a $\Delta$-wing or a tilted grim reaper. However, seeing as $W$ is a planar wing (see Definition \ref{def:wing-type}), this means that there exists a sequence of points $q_n \in W\cap \pi_y(t_n)$ with the following two properties:
\begin{enumerate}
    \item $q_n$ is a local maximum of $z$ in $\pi_y(t_n) \cap W$. In particular,
    $N(q_n)\perp \ee_1$, and
    \item $\langle q_n, \ee_3 \rangle - \langle p_n, \ee_3 \rangle \to \infty.$
\end{enumerate}
Reasoning as before,  $W-q_n \to \widehat W_\infty$ (up to a subsequence). Let $\widehat W$ be the connected component of $\widehat W_\infty$ containing $0$. Using
Property (2) above we have that the $z$-coordinate is not bounded below in $\widehat W \cap \pi_y(0).$ Then, using Proposition \ref{Wing-1}, we have that $\widehat W$ is a vertical plane. However, as $N(0) \perp \ee_1$ this plane would be not contained in the slab, which is absurd. This contradiction proves the claim.

Hence, by elementary topology, we get that $W^+(s)$ is a graph into the direction of $\ee_1$.
\end{proof}

In the remaining part of this section we study the case when $W$ is a wing of grim reaper type. Under this assumption we define the function
\begin{equation}\label{def:f_W}
s>t\mapsto f_W(s):=\displaystyle\min_{p\in\pi_y(s)\cap W} \langle p,\ee_3\rangle
\end{equation}
and the minimal axis of $W$ given by
\[\mathcal{M}_W:=\{p\in W\;:\;\langle p,\ee_3\rangle=f_W(\langle p,\ee_2\rangle)\}.\] About this set, we trivially have:

\begin{Lemma}\label{minimal-set}
$\mathcal{M}_W\subseteq\{\langle N,\ee_1\rangle=0\}.$
\end{Lemma}
Note that $\{\langle N,\ee_1\rangle=0\}$ is the union of smooth curves and the singular points are isolated, as seen by applying Theorem 2.5 in \cite{Che76} to the stability operator $L$ along with the fact that $u = \langle N,\ee_1\rangle$ is a Jacobi function $Lu=0$ (similarly to Remark 18 in \cite{Chi20}). Therefore, $\mathcal{M}_W$ has the same properties, too.

Given a point $(y,z)$ in the plane $\{x=0\}$, we define
\begin{equation}\label{def:L(y,z)}
L(y,z):=\{(x,y,z)\;:\;x\in\R\}.
\end{equation}
\begin{Lemma} \label{saturday}
For $t$ large enough, for all $(y,z)\in \Pi_{x}(W^+(t))$ such that $L(y,z)$ is transversal to $W$, the intersection $L(y,z) \cap W$ is either empty or consists of exactly two points.
\end{Lemma}
\begin{proof}
We proceed again by contradiction. Assume that there existed $t_n\nearrow+\infty$ and lines $L(t_n,z_n)\subset\pi_y(t_n)$, for some $z_n$, each intersecting $W$ in at least 4 points: namely, recall that the number of points in the intersection of $W$ with any such transversal line is finite, because each $W\cap \pi_y(t_n)$ is known to be bounded from below by Definition \ref{def:wing-type}. Furthermore this is always an even number, possibly zero.

This means that along the curve $W\cap \pi_y(t_n)$, the $z$-coordinate has a local maximum, $p_n\in W \cap \pi_y(t_n)$. 
By Theorem \ref{theorem-compactness}, possibly up to taking a subsequence, $W - p_n$ converges to some complete limiting translator $W_\infty\ni 0$. We let $W'$ denote the connected component of $W_\infty$ containing the origin, and notice that its Gauss map $N(0)\perp\ee_1$. So, by Lemma \ref{Wing-1}, $W'$ is a graph $z=f_\infty(x,y)$ over the $(x,y)$-plane. Using the smooth convergence, given a ball $B_R(0)\subseteq \R^3$, for $n$ large enough, each $(W-p_n)\cap B_R(0)$ is thus also a graph $z=f_n(x,y)$ over the disk of radius $R$ around $0$ in the $(x,y)$-plane.

But we have moreover, from $p_n$ being a local maximum, that
\begin{equation}\label{non-convex}
\frac{\partial^2 f_\infty}{\partial x^2}(0) = \lim_{n\to\infty} \frac{\partial^2 f_n}{\partial x^2}(0) \leq 0,
\end{equation}
for the graph describing $W'$. However $W'$ was either a tilted grim reaper or a $\Delta$-wing, where in both cases we know \eqref{non-convex} to not hold. This contradiction finishes the proof of the lemma, as the only possibilities which now remain are
\begin{equation}\label{cardinality}
\#\left(L(y,z) \cap W\right) \in \{0,2\}.
\end{equation}
\end{proof}

The next result is an analogue to Proposition \ref{planar-property} but for the grim reaper case.

\begin{Proposition}\label{Wing-decomposition}
If $t$ is large enough then $\mathcal{M}_W \cap W^+(t)$ is smooth. Moreover, $W^+(t)\setminus\mathcal{M}_W$ has two connected components and each connected component of this surface is a smooth graph (with boundary) into the direction of $\ee_1$.
\end{Proposition}
\begin{proof}
Take $s>t$, where $t$ is given by Lemma \ref{saturday}. We have that the $z$-coordinate has only a local minimum
and no local maxima on $\pi_y(s) \cap W.$

Furthermore, 
$L(s,f_W(s)) \cap W$ is a unique point, where $f_W$ is as in \eqref{def:f_W}. Otherwise, it would contain a segment. 
By analytic continuation, $W$ would contain the whole line $L(s,f_W(s)) $, which is
absurd.

Therefore, we can decompose $W^+(t)$ into three disjoint connected components $S_{W}^{+}\cup \mathcal{M}_W\cup S_{W}^{-}.$ Furthermore, $\mathcal{M}_W$ cannot have singularities in $W^+(t)$ and $S_W^{\pm}$ are smooth graphs into the direction of $\ee_1$.
\end{proof}

\section{Computing the entropy of translating solitons in slabs}\label{entropy}
The main goal of this section is to introduce an easy method for computing the entropies of complete, embedded translating solitons in slabs, with finite genus and entropy.

Before stating our result, we will need a definition.
\begin{Definition}\label{wing-type-numbers}
Let $\Sigma$ be a complete, embedded translating soliton of width $2w$, finite genus and finite entropy. Denote by $\omega^{-}_P(\Sigma)$ (respectively, $\omega^{+}_P(\Sigma)$) the number of left (respectively, right) planar type wings of $\Sigma$ and  $\omega^{-}_G(\Sigma)$ (respectively, $\omega^{+}_G(\Sigma)$) the number of left (respectively, right) grim reaper type wings of $\Sigma$.
\end{Definition}

The following main theorem makes this more precise, and links it to the entropy.

\begin{Theorem}\label{entropy-TGS}
Let $\Sigma$ be a complete, embedded translating soliton of width $2w<\infty$, finite genus and finite entropy.  Then $\lambda(\Sigma)$ is an integer. Furthermore,
\begin{equation}\label{ent.}
    \lambda(\Sigma)=\omega^-_P(\Sigma)+2\omega^-_G(\Sigma)=\omega^+_P(\Sigma)+2\omega^+_G(\Sigma).
\end{equation}
\end{Theorem}
\begin{proof}
Consider the family $\left\{\frac{1}{\sqrt[ ]{-t}}\Sigma-\sqrt[]{-t}\ee_3\right\}_{\{t<0\}}.$ Taking into account that $\Sigma$ has finite entropy, we get that for any sequence $\{t_n\}$ with $t_n\searrow-\infty$ the sequence $\left\{\frac{1}{\sqrt[ ]{-t_n}}\Sigma-\sqrt[]{-t_n}\ee_3\right\}$, up to passing to a subsequence, converges as varifolds to a self-shrinker. Actually, the self-shrinker limit will be the static plane $\pi_x(0)$ with finite multiplicity, because the boundary of the set $\mathcal{S}_w$ has $\pi_x(0)$ as limit. 

In order to get an upper bound for the multiplicity of this plane, we just need to remark that any horizontal line $L(p)$, where $p\in\pi_x(0)\setminus\{-s\leq y\leq s\}$, where $s$ is a fix number so that Proposition \ref{planar-property} and Proposition \ref{Wing-decomposition} hold, intersects the surface $\frac{1}{\sqrt[ ]{-t_n}}\Sigma-\sqrt[]{-t_n}\ee_3$ in at most $\omega^-_P(\Sigma)+2\omega^-_G(\Sigma)$ points if $p\in\pi_x(0)\cap\{y\geq s\}$ and  $\omega^+_P(\Sigma)+2\omega^+_G(\Sigma)$ points if $p\in\pi_x(0)\cap\{y\leq -s\}$, providing $t_n$ is sufficiently large. Thus, the multiplicity of $\pi_x(0)$ is less than $\min\{\omega^-_P(\Sigma)+\omega^-_G(\Sigma),\omega^+_P(\Sigma)+\omega^+_G(\Sigma)\}.$ Therefore, by using Huisken's monotonicity formula \cite{Hui90}, we deduce that (see Lemma 25 in \cite{Chi20} for more details)
\[\lambda(\Sigma)\leq \min\{\omega^-_P(\Sigma)+2\omega^-_G(\Sigma),\omega^+_P(\Sigma)+2\omega^+_G(\Sigma)\}.
\]

On the other hand,  we also aim to show that the entropy is bounded from below by  \[\max\left\{\omega^-_P(\Sigma)+2\omega^-_G(\Sigma),\:\omega^+_P(\Sigma)+2\omega^+_G(\Sigma)\right\}.
\] To see this, we intersect the left (respectively, right) wings with a plane $\pi_y(t)$ (respectively, $\pi_y(-t)$) where $t$ is taken large enough so that the planar type wings are graphs over the $yz-$plane and the grim reaper type wings are bi-graphs over the $yz-$plane. Namely Proposition \ref{saturday} and Proposition \ref{Wing-decomposition} ensure that we can choose such a $t$. Now, in this intersection, we have exactly, $\omega^+_P(\Sigma)+2\omega^+_G(\Sigma)$ (respectively, $\omega^-_P(\Sigma)+2\omega^-_G(\Sigma)$) arcs going to $+\infty$ with respect to the $z-$axis on the right (respectively, left) side of the graph. Consequently, when taking the limit
\(\Sigma-\tau_n\ee_3\to\Sigma_\infty,\) we see that the number of connected components is bounded from below by
\[
\max\{\omega^+_P(\Sigma)+2\omega^+_G(\Sigma),\: \omega^-_P(\Sigma)+2\omega^-_G(\Sigma)\}.
\]
At this point, we use the invariance under translations of the entropy to conclude that \[\max\{\omega^-_P(\Sigma)+2\omega^-_G(\Sigma),\omega^+_P(\Sigma)+2\omega^+_G(\Sigma)\}\leq\lambda(\Sigma).\] This completes the proof.
\end{proof}

Next, we will derive a few useful applications from this result. 

\begin{Corollary}\label{Wings-counting}
Let $\Sigma$ be as in Theorem \ref{entropy-TGS}. The wing numbers satisfy:
\begin{equation}
\begin{split}
&\omega^-_P(\Sigma)=\omega^+_P(\Sigma),\\
&\omega^-_G(\Sigma)=\omega^+_G(\Sigma),\\
&\omega(\Sigma)\equiv 0\pmod{2}.
\end{split}
\end{equation}
\end{Corollary}

\begin{proof}
We notice that Proposition \ref{weakly-infinity-limit} guarantees that for $\tau_n\nearrow +\infty$, the sequence $\Sigma+\tau_n \ee_3$ converges (up to a subsequence) to $\Sigma_\infty$ which is a finite family of vertical planes in $\mathcal{S}_w$ parallel to $\ee_2$, possibly with multiplicity.  For $t$ large enough, by Definition \ref{def:wing-type}, the number of planes in
$\Sigma_\infty \cap \{y<-t\}$ equals $\omega^-_P(\Sigma)$ while the number of planes in $\Sigma_\infty \cap \{y>t\}$ equals $\omega^+_P(\Sigma)$. Then, it is obvious that $\omega^+_P(\Sigma)=\omega^-_P(\Sigma).$

Now, finally, \eqref{ent.} means that also $\omega^-_G(\Sigma)=\omega^+_G(\Sigma),$ and so the number of wings $\omega(\Sigma)$ is always even.
\end{proof}

\begin{Definition}
The number \(
\omega_P(\Sigma):=\omega^-_P(\Sigma)+\omega^+_P(\Sigma)\) is called the number of planar type wings of $\Sigma$ and \( \omega_G(\Sigma):=\omega^-_G(\Sigma)+\omega^+_G(\Sigma)\) is called the number of grim reaper type wings of $\Sigma$.
\end{Definition}

\begin{Corollary}\label{components-estimate}
Let $\Sigma$ be a complete, embedded translating soliton of width $2w$, finite genus and finite entropy. Let $\{\tau_n\}$ be a sequence so that $\tau_n \nearrow+\infty$.
\begin{itemize}
    \item[(i)]If $\Sigma-\tau_n\ee_3\to\Sigma_{+\infty}$, then the number of connected components of $\Sigma_{+\infty}$ counted with multiplicity is exactly $\lambda(\Sigma);$
    \item[(ii)] If $\Sigma+\tau_n\ee_3\to\Sigma_{-\infty},$ then the number of connected components of $\Sigma_{-\infty}$ counted with multiplicity is exactly $\frac{1}{2}\omega_P(\Sigma).$
\end{itemize}
 
\end{Corollary}

\begin{Corollary}\label{entropy-estimate}
Let $\Sigma$ be a complete translating graph of width $2w.$ Then
\[
\lambda(\Sigma)\leq\left\lfloor \frac{w}{\pi}\right\rfloor+1,
\]
where $\left\lfloor x\right\rfloor$ denotes the largest integer smaller than or equal to $x\in\R$.
\end{Corollary}
\begin{proof}
Consider a sequence $\Sigma-\tau_n\ee_3 \to \Sigma_\infty$ as in Corollary \ref{components-estimate}. Recall that by Corollary \ref{Non-area-minimizing-1}, the distance between any two planes in $\Sigma_\infty$ is at least $\pi$. Hence the corollary follows since the multiplicity is one in the graphical case.
\end{proof}

\section{A gap for the width of translating\\ solitons with one end}\label{Structure}

This section is devoted to giving a gap theorem for the width of complete, connected translating solitons with one end. More precisely, we prove that if the soliton is not flat, then its width must be bigger than or equal to $\pi.$ This result extends a result due to Spruck and Xiao \cite{SX20} for vertical graphs. 

\begin{Proposition}\label{gap-width}
 Let $\Sigma$ be a complete, embedded connected translating soliton of width $2w>0$, finite genus and finite entropy with one end. Then ${\rm width}(\Sigma)\geq\pi.$
\end{Proposition}
\begin{proof}
As a tilted grim reaper type wing has width $\geq \pi$, then it suffices to prove $\omega_G(\Sigma)\geq1.$  Suppose this is not the case, then  $\omega_G(\Sigma)=0,$ which means that $\Sigma$ only possesses planar type wings. 

Then, by Proposition \ref{planar-property}, there exists $s$ large enough so that each connected component of $\Sigma^{\pm}(\pm s)$ is a graph with respect to $\ee_1.$ Consequently, the set $\{\langle N,\ee_1\rangle=0\}$ lies on $\Sigma[-s,s].$ In turn, Proposition \ref{weakly-infinity-limit} implies that cannot exist any arc of $\{\langle N,\ee_1\rangle=0\}$ which goes to $\pm\infty$ with respect to $z-$axis. Therefore, 
$$\{z \;: \; (x,y,z) \in \Sigma , \; \,\langle N(x,y,z),\ee_1\rangle=0\}$$
is a bounded set. Hence, there exists $R_0> 0$ large enough so that $\{\langle N,\ee_1\rangle=0\}$ lies in the horizontal solid cylinder $C_{R_0}:=\overline{B}_{R_0}(0)\times_{\ee_1}\R,$ where $\overline{B}_{R_0}$ is the closed ball centered at the origin in the $(y,z)$-plane and of radius $R_0.$ Notice that the intersection of the cylinder with the slab is bounded.

Our hypothesis of $\Sigma$ having one end ensures that  $\Sigma\setminus \p C_R$ is connected. Moreover $\partial C_R$ is transversal to $\Sigma$, for all $R>R_0$. Then,  $\Sigma\cap \p C_R$ is a Jordan curve, for all $R\geq R_0$. Moreover, we have that the projection 
$$\Pi_x: \Sigma \cap \partial C_R \rightarrow \{x=0\} \cap C_R$$
is a covering map of $\lambda(\Sigma)$-sheets. But this is impossible if $\lambda(\Sigma)\geq2,$ indeed in this case we would have a self-intersection. Therefore, we would have $\lambda(\Sigma)=1,$ and so, $\Sigma$ would be a plane, which is absurd.  This contradiction proves the proposition.
\end{proof}
It is an immediate consequence of the proof that
\begin{Corollary}
    \label{exist-g}
 Any complete, embedded connected translating soliton with one end of width $2w>0$, finite genus and finite entropy possesses at least two grim reaper type wings.
\end{Corollary}

If one assumes instead a lower bound on the $z$-coordinate, as well as width $\leq \pi$, then without assumptions on the topology, the (non-tilted) grim reaper cylinder $\mathcal{G}_0$ is known to be the unique finite entropy example \cite{IMR25}.

\section{Translators with finite width and $\lambda (\Sigma)<3$} \label{sec:chini}
As we have mentioned before, F. Chini \cites{Chi19, Chi20} classified the properly embedded translators with finite width and entropy less than $3$. In this section we are going to show how the techniques introduced in this paper allow us to provide a shorter proof of this theorem:
\begin{Theorem}[Chini \cites{Chi19, Chi20}] \label{l<3}
    Let $\Sigma$ be a simply connected translating soliton with finite width and $\lambda(\Sigma)<3.$ Then $\Sigma$ is one of the following solitons:
    \begin{enumerate}
        \item a vertical plane;
        \item a grim reaper cylinder (possibly tilted);
        \item a $\Delta$-wing.
    \end{enumerate}
\end{Theorem}
\begin{proof}
    From Theorem \ref{entropy-TGS}, we know that $\lambda(\Sigma)$ is a positive integer, so either $\lambda(\Sigma)=1$ or $\lambda(\Sigma)=2$.

    If $\lambda(\Sigma)=1$, then $\Sigma$ is a vertical plane.

    If $\lambda(\Sigma)=2$, then $\Sigma-z \ee_3$ converges (subsequentially) to $2$ planes, as $z \to +\infty.$ Hence, we have two possibilities for the number of wings (recall Definition \ref{def:wing-type}):
    \begin{enumerate}[(a)]
        \item $\omega(\Sigma)=4$ and all the wings are planar.
        \item $\omega(\Sigma)=2$ and both wings are of grim reaper type.
    \end{enumerate}
From Corollary \ref{exist-g}, Case (a) is not possible.

Next, we are going to prove that $\Sigma$ is mean convex. To do this, consider $F: \R^3 \rightarrow \R $
$$F(p)= \langle p, \ee_2 \rangle.$$
Theorem \ref{structure-theorem} implies that that $F|_\Sigma$ is a Rad\'o function in the sense of Definition \ref{rado-def}.

Consider $a<b$ regular values of $F|_\Sigma$ and a sequence  $z_n\nearrow +\infty$. Let us define
$$\Sigma(a,b):= \{ p \in \Sigma \; : \; a<F(p) <b\}. $$
  and
  $$\Sigma_n(a,b):= \Sigma(a,b) \cap \{z \leq z_n\}.$$
By Proposition \ref{weakly-infinity-limit}, we know that a subsequence of $\Sigma\pm z_n \ee_3$ converges (smoothly on compact sets) to vertical planes parallel to $\{x=0\}$.

Hence, up to a subsequence, we can assume that $F|_{\Sigma_n(a,b)}$ does not have critical points along
$\partial \Sigma_n(a,b)$. Moreover, we can also assume that
$$ F|_{\Sigma_n(a,b)} : \Sigma_n(a,b) \rightarrow (a,b)$$
is proper and 
it  is strictly monotone along
$\partial \Sigma_n(a,b). $
Hence, we can use Theorem~\ref{noncompact-intro-theorem} and Remark~\ref{re:8} to deduce that the number of critical points counted with multiplicity satisfies
\[
\mathsf{N}\left(F|_{\Sigma_n(a,b)}\right) \leq \beta+|Q|-\chi(\Sigma_n(a,b))-|A|=1+0-1-0=0.
\]
Taking into account that $\Sigma_n[a,b]$ represents an exhaustion of $\Sigma(a,b)$ we can apply Remark~\ref{re:semi} to deduce that
\[
\mathsf{N}\left(F|_{\Sigma(a,b)}\right) \leq \liminf_n\mathsf{N}\left(F|_{\Sigma_n(a,b)}\right) = 0,
\]
Taking limits as $a\to -\infty$ and $b\to +\infty$, and using Remark~\ref{re:semi} again, we get that
\[
\mathsf{N}\left(F|_{\Sigma}\right)=0
\]

Thus, Lemma \ref{gauss-arc}
would imply that $N(\mathcal{H})=\varnothing$ and then $\Sigma$ is mean convex.   As $\Sigma$ has finite width, then Theorem \ref{Classification} implies that $\Sigma$ is either a tilted grim reaper, or a $\Delta$-wing.
  
\end{proof}

\begin{Corollary}\label{one-wing}
Let $\Sigma$ be a complete, embedded, simply connected translating soliton of width $2w>0$ and finite entropy. Assume that $\omega^{+}(\Sigma) = 1$ (or equivalently $\omega^{-}(\Sigma) = 1$). Then, $\Sigma$ is either a tilted grim reaper surface or a $\Delta-$wing. 
\end{Corollary}

\begin{proof} First, notice that, as the width is positive, $\Sigma$ is not flat. Furthermore, from Theorem \ref{entropy-TGS} we have that $\lambda(\Sigma) \leq 2.$ Then, we only have to use Theorem \ref{l<3} to conclude the result.
\end{proof}

\begin{Corollary}\label{Classification-slab-3pi}
The unique complete connected translating graphs in $\R^3$ of width $2w$, where $2w\in[\pi,2\pi)$, are tilted grim reaper surfaces, and $\Delta$-wings.
\end{Corollary}
\begin{proof}
By Corollary \ref{entropy-estimate} we know that
$\lambda(\Sigma)\leq2.$ Hence, $\lambda(\Sigma)=2$, but then Theorem \ref{l<3} implies the result.
\end{proof}

\section{Uniqueness of the pitchforks}\label{pitchforks-section}

The main goal of this final section consists in characterizing pitchforks as the unique examples of complete, embedded, simply connected translating soliton, $\Sigma \subset \R^3$ of finite width, entropy $\lambda(\Sigma) = 3$ under additional hypothesis.

We begin with the study of some geometric properties of the set $\mathcal{H}$ that will help us to derive important results about $\Sigma$.

\begin{Lemma}\label{four_components}
Let $\Sigma$ be a complete embedded simply connected translating soliton of finite entropy and finite width with $\omega(\Sigma) = 4$. Then
\[
N|_{\mathcal{H}}: \mathcal{H}\to E
\]
is a diffeomorphism onto its image. Furthermore, for $q\in\mathcal{H}$ arbitrary, $\Sigma\setminus T_q\Sigma$ has precisely four connected components, two on each side of $T_q\Sigma.$
\end{Lemma}
\begin{proof}
Given $\nu\in E \setminus \{-\ee_1, \ee_1\}$, we consider the Rad\'o{} function
\begin{eqnarray}
&f_\nu: \Sigma \to \R,\\
&p\mapsto \langle p,\nu\rangle.
\end{eqnarray}
Consider two regular values, $a < b$, of $f_\nu$ and define
\[\Sigma_\nu(a,b):=\{ p \in \Sigma \; : \; f_\nu(p) \in(a,b) \} \]
and 
\[\Sigma_\nu^n(a,b):=M(a,b) \cap \{|z|\leq z_n\}, \quad z_n \nearrow \infty\]
By Proposition \ref{weakly-infinity-limit}, we know that a subsequence of $\Sigma\pm z_n \ee_3$ converges (smoothly on compact sets) to vertical planes parallel to $\{x=0\}$.
Then, up to a subsequence, we can assume that $f_\nu$ does not have critical points along
$\partial \Sigma_\nu^n(a,b)$. Moreover, we can also assume $f_\nu$ is strictly monotone along
$\partial \Sigma_\nu^n(a,b). $
Hence, we can use Theorem~\ref{th-one} and Remark~\ref{re:8} to estimate  the number of critical points (counted with multiplicity.) As $f_\nu$ is strictly monotone along
$\partial \Sigma_\nu^n(a,b), $ then the number of minima and maxima of $f_\nu$ along $\partial \Sigma_\nu^n(a,b) $ is zero. So $|Q|=|A|=0$. On the other hand, the number of points in $\{f_\nu=t\} \cap \partial \Sigma_\nu^n(a,b)$, for $t>a$ is $4$. So, the number $$\beta=\lim_{t\to a, t>a} \sharp (\{f_\nu=t\} \cap \partial \Sigma_\nu^n(a,b))=4.$$
Hence,
\[
\mathsf{N}\left(f_{\nu}|_{\Sigma_\nu^n(a,b)}\right) \leq \frac 12 \beta+|Q|+|A|-\chi(\Sigma_\nu^n(a,b))=1.
\]
Notice that $\Sigma_\nu^n(a,b)$ is always a disk, and so $\chi(\Sigma_\nu^n(a,b))=1.$

Taking into account that $\Sigma_\nu^n(a,b)$ represents an exhaustion of $\Sigma_\nu(a,b)$ we can apply Remark~\ref{re:semi} to deduce that
\[
\mathsf{N}\left(f_{\nu}|_{\Sigma_\nu(a,b)}\right) \leq \liminf_n\mathsf{N}\left(f_{\nu}|_{\Sigma_\nu^n(a,b)}\right) = 1,
\]
Taking limits as $a\to -\infty$ and $b\to +\infty$, and using Remark~\ref{re:semi} again, we get that
\[
\mathsf{N}\left(f_{\nu}\right) \leq 1,
\]
If $\mathsf{N}\left(f_{\nu}\right)=0,$ then Lemma \ref{gauss-arc}
would imply that $N(\mathcal{H})=\varnothing$ and then $\Sigma$ would be a vertical graph, which is impossible by Theorem \ref{Classification} if we recall that we are assuming $\omega(\Sigma) = 4$. Thus:

\begin{equation}\label{one_preimage}
\forall \nu\in E\setminus \{-\ee_1, \ee_1\}: \#\left(N^{-1}(\nu)\right) \in \{0,1\},
\end{equation}
proving injectivity of the Gauss map when restricting to $\mathcal{H} = N^{-1}(E)$. Then if $q\in\mathcal{H}$, equivalently if $\nu:=N(q)\in E\setminus \{-\ee_1, \ee_1\}$, we conclude that
\[
\Sigma \setminus \pi = \Sigma \setminus T_q\Sigma.
\]
has two connected components on each side of $\pi$, as otherwise there would be another point $q'$ with $N(q') = \pm N(q)$, which is impossible by the count of critical points (i.e. preimages) in \eqref{one_preimage}.
\end{proof}
\begin{figure}[htbp]
\begin{center}
\includegraphics[width=.66\textwidth]{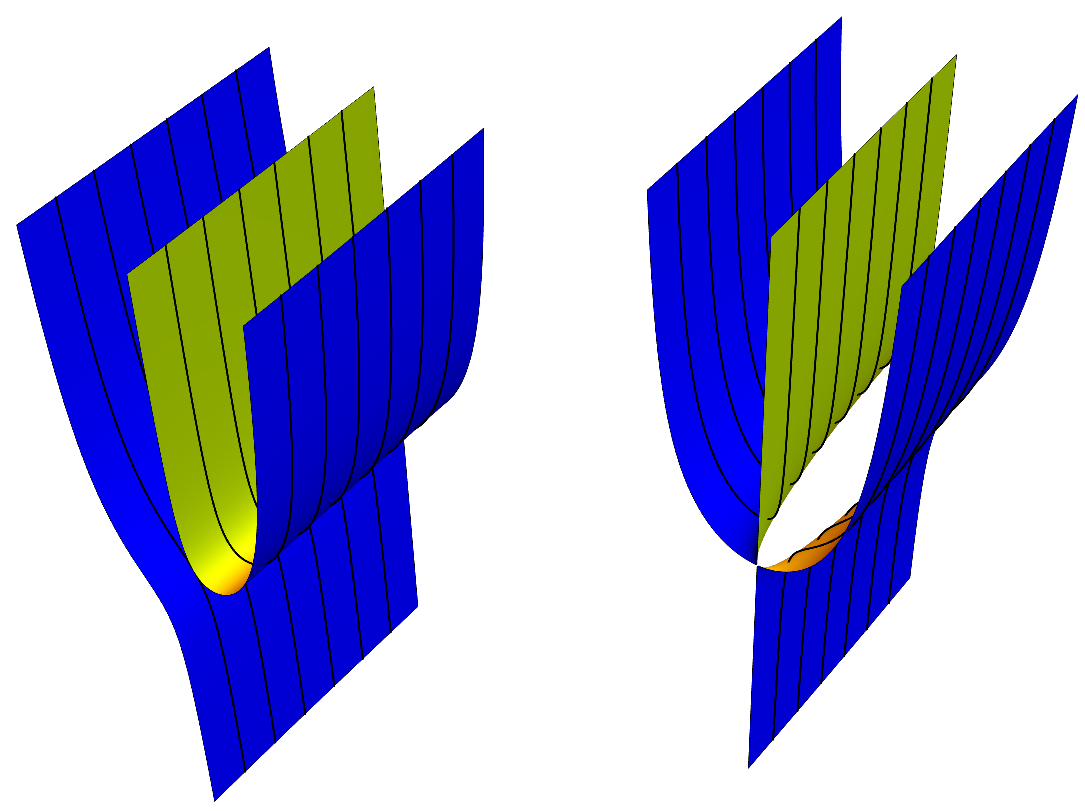}
\caption{\small On solitons as in Lemma \ref{four_components}, for each $q\in\mathcal{H}$ arbitrary, $\Sigma\setminus T_q\Sigma$ has precisely four connected components, two on each side of $T_q\Sigma.$}
\label{uncapped}
\end{center}
\end{figure}

\begin{Corollary}\label{half-arc}
Let $\Sigma$ be as in Lemma \ref{four_components}. Then, up to a possible change of orientation of the surface,
\[
N(\mathcal{H}) = E\cap \{y>0\}.
\]
\end{Corollary}

 Using these result we are now going to that if we also impose on the translating soliton that it contains a vertical line, then the soliton must be a pitchfork. That is the purpose of the next theorem.

\begin{Theorem}\label{pitchforks-line-theorem}
Let $\Sigma$ be a complete, embedded, simply connected translating soliton of finite width and entropy $\lambda(\Sigma) = 3$. If $\Sigma$ contains a vertical straight line, $A$, then $\Sigma$ is a pitchfork.
\end{Theorem}
\begin{proof}
Denote by $\Sigma$ any complete simply connected translating of width $2w>0$ and entropy $\lambda(\Sigma) = 3$. We first note that Corollary \ref{Wings-counting} and Corollary \ref{exist-g} imply that $\Sigma$ possesses one grim reaper type wing and one planar type wing on both the left and right sides. In particular, the total number of wings is $\omega(\Sigma)=4.$

From Lemma \ref{four_components}, we conclude that
\begin{equation} \label{counting-critical-point}
\mbox{ $\forall \nu \in E \setminus \{-\ee_1,\ee_1\}: f_{\nu}$ has exactly one critical point.}
\end{equation}
Up to a translation, we can assume that the unique critical point of $f_{\ee_2}$ is $p_0=(x_0,0,0)$, for some $x_0\in\mathbb{R}$. Moreover, up to a possible change of orientation of $\Sigma$, we may assume that $N(p_0)=\ee_2$. Then, by Lemma \ref{gauss-arc} we know that $$N(\mathcal{H})=E \cap \{y>0\}.$$ Next, we would like to argue that
$$\mathcal{H}:=\{H=0\}=A = \{(x_0,0)\}\times \R.$$

As $\Sigma$ is a $g$-minimal surface and $A$ is a geodesic with respect to the metric $g$, then we can use the Schwarz reflection principle to ensure that $A$ is a line of symmetry of $\Sigma$.  Then, $A$ passes through $p_0$, otherwise there would be more than two points whose normal is $\pm \ee_2.$ Moreover, it is trivial that $A \subseteq \mathcal{H}.$

As the symmetry with respect to $A$ preserves the slab $\mathcal{S}_w$, then $p_0=(0, 0, 0)$ and $A$ coincides with the $z$-axis.

We have that  $N(0,0,z) \to \pm \ee_1$, as $z \to \pm \infty$ (recall that the subsequential limits of $\Sigma-z \ee_3$ are vertical planes parallel to $\{x=0\}$). Then, the injectivity of $N$ along $\mathcal{H}$ implies that $N(A)=E\cap\{y>0\}.$
Combining this fact again  with the injectivity of $N$ along $\mathcal{H}$ (Lemma \ref{four_components}), we conclude that the closed set $\mathcal{H}:=\{H=0\}$ consists of just one one-dimensional smooth curve without any singularities, and that in fact equality of these sets holds, as desired: $\mathcal{H} =A$.

We have thus proved that $\Sigma\setminus(\{(0,0)\} \times \R)$ has two connected components and each of these components is strictly mean convex. We aim to show that, in fact, the projection $\Pi_z$ restricted to each connected component is a diffeomorphism onto its image. Otherwise, let $p_1$ and $p_2$ so that $\Pi_z(p_1)=\Pi_z(p_2)$. So, let us consider an  arc $\alpha$ of the curve $\pi_y(\langle p_1,\ee_2\rangle) \cap \Sigma$ connecting $p_1$ and $p_2$. This would yield either a local maximum or local minimum for the function $\mathbf{x}\mapsto \langle \mathbf{x},\ee_1\rangle$. But at such extremum point, the projection $\Pi_z$ restricted to any neighborhood of the point could not be one-to-one, and this gives a contradiction. Therefore, each connected component of  $\Sigma\setminus(\{(0,0)\}\times \R)$ is a smooth graph. Now, the theorem follows from the classification theorem for semi-graphical translators, see Theorem \ref{semigraphical-theorem} or \cite[Theorem 34]{HMW22-2}, seeing as all the other examples in the list are periodic, hence of infinite entropy, and/or have infinite width.
\end{proof}

\begin{Remark}
Although it has been conjectured in \cite{HMW22-2} that, given $w \geq \pi$, there is only one (up to rigid motions) pitchfork of width $2w$, it has at the time of writing not been rigorously proven. See Remark \ref{pitchfork-uniqueness} for further explanations.
\end{Remark}

\begin{Remark}
The hypothesis of $\Sigma$ being simply connected cannot be removed, as there are examples, constructed by X.H. Nguyen \cite{Ngu09}, of complete translators contained in a slab, with entropy $\lambda(\Sigma) = 3$ and infinite genus (see also \cite{HMW22-2}.) Nguyen's examples form a $1$ parameter family of translators 
$$\{M_a \; : \; a \in (0,\infty)\}.$$
The translator $M_a$ is contained in the slab $\mathcal{S}_{w(a)}$, where $w(a)\in (\pi,2 \pi)$ is an increasing function of $a$. Moreover, it contains the family of vertical straight lines $\{x=0, y= n a\}; $  $n \in \z.$ So, it is periodic with period equal to $(0,2 a,0),$ (see Fig. \ref{sequence2}.)
\end{Remark}

\appendix
\section{Results on $g$-area minimizing surfaces}\label{Results-g-area}
The main goal of this appendix is to show that all complete connected translating graphs in general direction are proper surfaces in $\R^3$ and they have local area bounds. Both results rely on the $g$-area minimizing properties of self-translating graphs. 

Notice that it will suffice (by Proposition \ref{Plane-Case}) to work with complete graphs $$\Sigma \subseteq \mathcal{S}_w =  \{(x,y,z) \in \R^3 \; :\; |x|<w\}$$ which are in {\bf Case II}, i.e.
$v \in \Ss^2\cap \{x=0\}.$ This  implies that
\begin{equation}\label{v_theta}
    v=v_\theta:=\sin\theta\ee_3+\cos\theta\ee_2,
\end{equation}
for some $\theta\in\left[0,\pi\right)$. From now on, we thus normalize the vector $v$ to always be in the family $\{v_\theta:=\sin\theta\ee_3+\cos\theta\ee_2\;:\; \theta\in\left[0,\pi\right)\}.$ Moreover, we will use the family of planes $\{P_\theta:=[v_\theta]^{\perp}\},$ where $\theta\in[0,\pi)$, and the functions from Definition \ref{def:graphs} defined on domains $\Omega_\theta\subseteq P_\theta$ will then be labeled $u_\theta$.

We begin this part by introducing some results about the existence of $g-$area minimizing surfaces, see \cites{Fed69, HS79, Sim83} for details.

\begin{Theorem}\label{Least-area}
 Let $\gamma_1$ and $\gamma_2$ be two disjoint smooth Jordan curves in $\R^3$. Assume that there is a locally integral 2-dimensional current $C$ so that $\p C=[\gamma_1]\cup[\gamma_2]$ and ${\bf M}_g[C]<+\infty,$ where ${\bf M}_g[C]$ denotes the mass of $C$ in $(\R^3,g).$ Then, there is a $g-$area minimizing surface $S$ with boundary $\gamma_1\cup\gamma_2.$ Furthermore, $S$ is smooth up to the boundary. 
\end{Theorem}

Using this result and the $g$-area minimizing property in Lemma \ref{Area-minimizing}, we are now going to prove that each complete translating graph is proper.

\begin{Proposition}\label{Proper-condition}
Let $\Sigma^2\subseteq \R^3$ be a complete connected translating graph. Then $\Sigma$ is proper in $\R^3.$
\end{Proposition}
\begin{proof}
We are assuming that $\Sigma={\rm Graph}[u_\theta],$ where $u_\theta:\Omega_\theta\to\R$ is a smooth function. Before we start the proof, we will introduce some notation. We already know by Lemma \ref{domain} that the boundary of the solid cylinder $\Omega_\theta\times_\theta\R$ is formed by tilted grim reapers (tangent to $v_\theta$) and at most two vertical planes. 

Let $M$ be a connected component of $\p (\Omega_\theta\times_\theta\R).$
Suppose $M=\pi_x(t)$  In this case, we define 
\[
M(s):=\left\{
\begin{array}{rcl}
M-s\ee_1 &if& \Sigma\subseteq\{x<t\} \\
M+s\ee_1 &if& \Sigma\subseteq\{x>t\} \\
\end{array}
\right.
\] 

If $M$ is a tilted grim reaper tangent to $v_\theta$, then  we define
\[
M(s):=\left\{
\begin{array}{rcl}
M-s\ee_3 &if& \Sigma\subseteq {\rm Nconv}(M) \\
M+s\ee_3 &if& \Sigma\subseteq {\rm Conv}(M)\\
\end{array}
\right.,
\]
where ${\rm Nconv}(M)$ indicates the non-convex region in $\R^3$ with boundary $M$, and ${\rm Conv}(M)$  denotes the convex one. 

Then, for $s$ small enough, we define $C(s)$ as the connected component of $\Omega_\theta\times_\theta\R$ whose closure only intersects \[\bigcup_{M\subset\p (\Omega_\theta\times_\theta\R)}M(s).\]

Once we have introduced this notation, we may return to the proof. In order to conclude that $\Sigma$ is proper, we only need to verify $\overline{B_r(0)}\cap \Sigma$ is compact for all $r>0$. So, let $B_r(0)$ be an open ball in $\R^3$. Without loss of generality, it will be assumed that $\p B_r(0)\pitchfork\Sigma.$ 

We claim that $B_r(0)\cap\Sigma$ has finitely many connected components. We argue by contradiction, suppose that $\Sigma\cap B_r(0)$ has infinitely many connected components. For all $s>0$ sufficiently small, we may conclude that $\overline{B_r(0)\cap C(s)}$ is a compact set in $\Omega_\theta\times_\theta\R.$ Thus, for all $s>0$ sufficiently small the $g-$area minimizing condition of $\Sigma$ in $\Omega_\theta\times_\theta\R$ implies that the number of leaves of  $\Sigma\cap B_r(0)$ intersecting $\overline{B_r(0)\cap C(s)}$ is finite. This yields that the leaves of $\Sigma\cap B_r(0)$ shall accumulate on boundary of $\Omega_\theta\times_\theta\R.$

Therefore, if $\{S_i\}$ is a sequence of distinct connected components of $\Sigma\cap B_r(0)$, then any limit of a convergent subsequence of such sequence shall belong to $B_r(0)\cap\p(\Omega_\theta\times_\theta\R).$ Passing to a subsequence, we will assume that $\{S_i\}$ converges to $B_r(0)\cap\p(\Omega_\theta\times_\theta\R).$

Let $i_0$ be large enough so that if $i\geq i_0$ it holds
\[
\mathcal{H}_g^2(S_i)+\mathcal{H}_g^2(S_{i+1})>\frac{3}{2}\mathcal{H}_g^2(B_r(0)\cap\p(\Omega_\theta\times_\theta\R))
\]

\[
\mathcal{H}_g^2(\widehat{C}_i)<\frac{1}{2}\mathcal{H}_g^2(B_r(0)\cap\p(\Omega_\theta\times_\theta\R)),
\]
and
\[
{\rm dist}_{H}(\p S_i,\p S_{i+1})<\pi
\]
where $\widehat{C}_i$ indicates the annulus in $\p B_r(0)$ with boundary $\p S_i$ and $\p S_{i+1}$, and ${\rm dist}_{H}(\cdot,\cdot)$ indicates the Hausdorff distance in $\R^3$ with the euclidean metric. These assumptions and Theorem \ref{Least-area} ensure that for all $i\geq i_0$ there is a $g$-area minimizing surface $C_i$ with boundary $\p S_i$ and $\p S_{i+1}$ in $(\R^3,g)$ so that

\begin{equation}\label{main-annulus-estimate}
\mathcal{H}_g^2(C_i)\leq\mathcal{H}_g^2(\widehat{C}_i)< \mathcal{H}_g^2(S_i)+\mathcal{H}_g^2(S_{i+1})
\end{equation}

Notice that $C_i$ is connected. Indeed, if $C_i$ was disconnected, then $\p S_i$ and $\p S_{i+1}$ belongs to disjoint connected components. 

Let $M_i$ be a connected component of $C_i$ with boundary $\p S_i.$ Using translating of the  connected components on boundary of $\p(\Omega_\theta\times_\theta\R),$ it may conclude that $M_i$ belongs to $\Omega_\theta\times_\theta\R.$ Since, $\Sigma$ is $g-$area minimizing in this set, we conclude that $\mathcal{H}_g^2(S_i)\leq \mathcal{H}_g^2( M_i).$ Analogously, it may be deduced that the connected component of $C_i$ with boundary $\p S_{i+1}$, denoted by $M_{i+1}$, satisfies $M_{i+1}\subseteq\Omega_\theta\times_\theta\R$ and $\mathcal{H}_g^2(S_{i+1})\leq \mathcal{H}_g^2( M_{i+1}).$ In particular, it holds
\begin{eqnarray}
\nonumber\mathcal{H}_g^2(S_i)+\mathcal{H}_g^2(S_{i+1})&\leq&\mathcal{H}_g^2( M_{i})+\mathcal{H}_g^2( M_{i+1})\\
\nonumber&\leq&\mathcal{H}_g^2(C_i)\leq\mathcal{H}_g^2(\widehat{C}_i)\\
\nonumber&<& \mathcal{H}_g^2(S_i)+\mathcal{H}_g^2(S_{i+1}),
\end{eqnarray}
absurd.

We also remark that $C_i$ belongs to $\Omega_\theta\times_\theta\R.$ We can see this by using that the connected components of $\p(\Omega_\theta\times_\theta\R)$ as barriers as before.

Fix any $i\geq i_0$ and let $s$ be small enough so that $\overline{S_i}$ and $\overline{S_{i+1}}$ belong to $C(s),$ and so $C_i$ belongs to $C(s)$ too. Denote by $\widehat{\Sigma}$ the piecewise surface obtained from $\Sigma$ by removing the sets $S_i$ and $S_{i+1}$ and adding $C_i$ in their place. From \eqref{main-annulus-estimate}, one gets

\begin{eqnarray}
\nonumber\mathcal{H}_g^2(\widehat{\Sigma}\cap \overline{B_r(0)\cap C(s)})&\leq&\mathcal{H}_g^2(\widehat{C}_i)+\mathcal{H}_g^2(\Sigma\cap \overline{B_r(0)\cap C(s)}\setminus \{S_i\cup S_{i+1}\})\\
\nonumber&<&\mathcal{H}_g^2(\Sigma\cap \overline{B_r(0)\cap C(s)}),
\end{eqnarray}
Namely, this contradicts the fact that $\Sigma$ is $g-$area minimizing in $\Omega_\theta\times_\theta\R.$ 

Consequently, we have shown that $\Sigma\cap B_r(0)$ has finitely many connected components. Now it is easy to infer that $\Sigma\cap \overline{B_r(0)}$ is compact, which completes the proof.
\end{proof}

Using the Proposition \ref{Proper-condition} and Lemma \ref{Area-minimizing}, it is then straightforward to obtain the following local area bounds for complete translating graphs.
\begin{Lemma}\label{Locally-Area}
Let $\Sigma$ be a complete connected translating graph. Then, for all $q\in \R^3$ and $r>0$ it holds:
\[
\mathcal{H}_g^2(B_r(q)\cap\Sigma)\leq\mathcal{H}_g^2(\p B_r(q)).
\]
\end{Lemma}

Next, we would like to conclude that some surfaces cannot be $g-$area minimizers. For that, we need the following result.

\begin{Lemma}\label{Non-area-minimizing}
Consider the planes $\pi_x(-c)$ and $\pi_x(c)$, where $c>0$. If $2c<\pi$, then $\pi_{x}(-c)\cup\pi_{x}(c)$ is not $g-$area minimizing in $\R^3.$
\end{Lemma}
\begin{proof}
Consider for $h>0$ the rectangles
\[D_\pm:=\{(\pm c,y,z)\;:\;y\in[0,h]\ {\rm and}\ z\in[-\log\cos c,\sigma-\log\cos c]\},\]
where $\sigma$ satisfies
\(
\sin c<\tanh \sigma.
\) In \cite[Proposition 3.11]{GHLM20}, it was proved that if $h$ is sufficiently large, then there is a piecewise (in fact, smooth) cylinder $\mathcal{C}$ in the set between $\pi_{x}(-c)$ and $\pi_{x}(c)$ such that $\p \mathcal{C}=\p D_-\cup\p D_+$ and $\mathcal{H}_g^2(\mathcal{C})<\mathcal{H}_g^2(D_-)+\mathcal{H}_g^2(D_+),$ where $\mathcal{H}^2_g(\cdot)$ denotes the two-dimensional Hausdorff measure associated to Ilmanen's metric on $\R^3$. Thus, if $\Sigma':=(\pi_{x}(-c)\setminus D_-)\cup\mathcal{C}\cup(\pi_{x}(c)\setminus D_+),$ then 
\[
\mathcal{H}_g^2(\Sigma'\cap B_R)<\mathcal{H}_g^2((\pi_{x}(-c)\cup\pi_{x}(c))\cap B_R), \forall\ R\ {\rm large\ enough}, 
\] 
where $B_R=B_R(0)=\{p\in\R^3\;:\;|p|<R\}$. Therefore, $\pi_{x}(-c)\cup\pi_x(c)$ is not $g-$area minimizing in $\R^3.$
\end{proof}

\begin{Corollary}\label{Non-area-minimizing-1}
Let $\Sigma$ be a complete translating soliton that possesses two disjoint parallel planes with distance less than $\pi$ as connected components. Then, $\Sigma$ is not $g-$area minimizing in $\R^3$.
\end{Corollary}

\bibliographystyle{amsplain, amsalpha}

\begin{thebibliography}{60}

\bibitem{AW94} S. J. Altschuler and  L.F. Wu, Translating surfaces of the non-parametric mean curvature flow with prescribed contact angle, Calc. Var. Partial Differential Equations 2 (1994), no. 1, 101--111.

\bibitem{Che76} S. Y. Cheng, Eigenfunctions and nodal sets, Comment. Math. Helv. 51, 43--55 (1976).

\bibitem{Chi19} F. Chini,
Some classification results for translating solitons and ancient mean curvature flows. PhD thesis, University of Copenhagen (2019).

\bibitem{Chi20} F. Chini, Simply connected translating solitons contained in slabs. Geom. Flows 5.1 (2020): 102--120.

\bibitem{CM19} F. Chini and N. M. M{\o}ller,
Ancient mean curvature flows and their spacetime tracks, arXiv:1901.05481 (2019).

\bibitem{CM21} F. Chini and N. M. M{\o}ller, Bi-Halfspace and Convex Hull Theorems for Translating Solitons. International Mathematics Research Notices, Volume 2021, Issue 17, September 2021, 13011--13045, doi.org/10.1093/imrn/rnz183

\bibitem{CSS07} J. Clutterbuck, O. Schn{\"u}rer and F. Schulze, Stability of translating solutions to mean curvature flow. Calc. Var. and Partial Differential Equations 2 (2007):281--293.

\bibitem{CM12}T. H. Colding and W. P. Minicozzi II, Generic mean curvature flow I: generic singularities. Ann. of Math.(2) 175.2 (2012): 755--833.

\bibitem{DDN17} J.  D\'avila, , M. Del Pino, and X. H. Nguyen, Finite topology self-translating surfaces for the mean curvature flow in $\R^3$. Advances in Mathematics, 320, 674-729 (2017). https://doi.org/10.1016/j.aim.2017.09.014

\bibitem{EM16} M. Eichmair and J. Metzger, Jenkins-Serrin-type results for the Jang equation. J. Differential Geom. 102.2 (2016): 207--242.

\bibitem{Fed69} H. Federer, Geometric measure theory, Die Grundlehren der mathematischen Wissenschaften, Band 153, Springer-Verlag New York Inc., New York, 1969.

\bibitem{GHLM20} E. S. Gama, E. Heinonen, J. H. Lira, and F. Mart\'in, The Jenkins-Serrin problem for translating horizontal graphs in $M\times\R.$ Rev. Mat. Iberoam. 37.3 (2020): 1083--1114.

\bibitem{GMM25} E. S. Gama, F. Mart\'{\i}n and N. M. M\o{}ller, Uniqueness of tangent planes at infinite time for collapsed self-translating solitons. Work in progress.

\bibitem{HS79} R. Hardt and L. Simon. Boundary Regularity and Embedded Solutions for the Oriented Plateau Problem. Ann. of Math.(2) 110.3 (1979): 439--486.

\bibitem{HIMW19} D. Hoffman, T. Ilmanen, F. Mart\'in and B. White, Graphical translators for the mean curvature flow.  Calc. Var. Partial Differential Equations. 58.117 (2019): https://doi.org/10.1007/s00526-019-1560-x.

\bibitem{HMW22-1} D. Hoffman, F. Mart\'in and B. White, Scherk-like translators for mean curvature flow. J. Differential Geom. vol. 122 (2022), no. 3, 421--465.

\bibitem{HMW22-2} D. Hoffman, F. Mart\'in and B. White, Nguyen's Tridents and the Classification of semigraphical translators for mean curvature flow. J. Reine Angew. Math. vol. 786 (2022), 79--105.

\bibitem{HMW23} D. Hoffman, F. Mart\'in and B. White, Morse-Rad\'o Theory for Minimal Surfaces. J. Lond. Math. Soc. (2) vol. 108 (2023), no. 4, 1669--1700.

\bibitem{HMW24} D. Hoffman, F. Mart\'in and B. White, Translating annuli for mean curvature flow. Adv. Math. vol. 455 (2024), Paper No. 109875, 109 pp.

\bibitem{Hui90} G. Huisken, Asymptotic behavior for singularities of the mean curvature flow. J. Differential Geom. 31.1 (1990): 285--299.

\bibitem{Ilm94} T. Ilmanen, Elliptic regularization and partial regularity for motion by mean curvature. Mem. Am. Math. Soc. vol. 108 (1994): 1--94.

\bibitem{IMR25}D. Impera, N. M. M\o{}ller and M. Rimoldi, Rigidity and non-existence results for collapsed translators, Int. Math. Res. Not. 2025 (2025), no.7. https://doi.org/10.1093/imrn/rnaf076

\bibitem{IR19}D. Impera, and M. Rimoldi,
Quantitative index bounds for translators via topology. Math. Z. 292, 513--527 (2019). https://doi.org/10.1007/s00209-019-02276-y

\bibitem{KKM18} N. Kapouleas, S. J. Kleene and N. M. M{\o}ller, Mean curvature self-shrinkers of high genus: non-compact examples. 
J. Reine Angew. Math. 739 (2018), 1--39. 

\bibitem{KP22} D. Kim and J. Pyo, Properness of translating solitons for the mean curvature flow. Internat. J. Math. {33} (2022), no. 4, Paper No. 2250032, 6 pp. 

\bibitem{KS19} K. Kunikawa and S. Saito, Remarks on topology of stable translating solitons. Geometriae Dedicata (2019) 202:1--8.

\bibitem{MSTW24} F. Mart{\'i}n, M. S\'aez, R. Tsiamis and B. White, Classification of semigraphical translators.  Preprint arXiv:2411.16889, 2024.

\bibitem{MSHS15} F. Mart{\'i}n, A. Savas-Halilaj and K. Smoczyk, On the topology of translating solitons of the mean curvature flow. Calc. Var. Partial Differential Equations. 54.3 (2015): 2853--2882.

\bibitem{Mol14} N. M. M\o{}ller,  Non-existence for self-translating solitons, arxiv:1411.2319 (2014).

\bibitem{Ngu09} X. H. Nguyen, Translating tridents, Comm. Partial Differential Equations 34 (2009), no. 1-3, 257-280.

\bibitem{Sha15} L. Shahriyari, Translating graphs by mean curvature flow. Geom. Dedicata. 175.1 (2015): 57--64.

\bibitem{Sim83} L. Simon, Lectures on geometric measure theory. Proc. Centre Math. Analysis, Austral.Nat. Univ., 3, Austral. Nat. Univ., Canberra, 1983.

\bibitem{Sol89} B. Solomon, On foliations of $\R^{n+1}$ by minimal hypersurfaces. Comment. Math. Helv. 61 (1989): 67--83.	

\bibitem{SX20} J. Spruck and L. Xiao, Complete translating solitons to the mean curvature flow in $\R^3$ with nonnegative mean curvature. Amer. J. Math. 142.3 (2020): 993--1015.

\bibitem{Wan11} X.- J. Wang, Convex solutions of the mean curvature flow, Ann. Math. (2) 173 (2011), 1185--1239.

\bibitem{Whi18} B. White, On the Compactness Theorem for Embedded Minimal Surfaces in 3-manifolds with Locally Bounded Area and Genus. Communications in Analysis and Geometry, Volume 26, Number 3, 659--678, 2018.

\bibitem{Whi21}  B. White, Mean curvature flow with boundary. Ars Inveniendi Analytica (2021), Paper No. 4, 43 pp. DOI 10.15781/vks5-5e33.

\bibitem{Whi22} B. White, The Boundary Term in Huisken's Monotonicity Formula and the Entropy of Translators. Comm. Anal. Geom. vol. 33 (2025), no. 1, 1--16.

\end{thebibliography}

\end{document}